\documentclass[11pt,reqno]{amsart}
\usepackage{array,listings,xcolor,tocvsec2}
\usepackage{amssymb,amsmath,amsthm,mathrsfs,enumerate,color,multirow}
\usepackage[justification=centering]{caption}
\usepackage[edge-length=0.5 cm]{dynkin-diagrams}
\usepackage[breaklinks=true]{hyperref}
\usepackage[letterpaper]{geometry}
\usepackage{pgf,tikz,amscd}
\usetikzlibrary{arrows}
\definecolor{codegreen}{rgb}{0,0.6,0}
\definecolor{codegray}{rgb}{0.5,0.5,0.5}
\definecolor{codepurple}{rgb}{0.58,0,0.82}
\definecolor{backcolour}{rgb}{0.95,0.95,0.92}

\lstdefinestyle{mystyle}{
    backgroundcolor=\color{backcolour},   
    commentstyle=\color{codegreen},
    keywordstyle=\color{magenta},
    numberstyle=\tiny\color{codegray},
    stringstyle=\color{codepurple},
    basicstyle=\ttfamily\footnotesize,
    breakatwhitespace=false,         
    breaklines=true,                 
    captionpos=b,                    
    keepspaces=true,                 
    numbers=left,                    
    numbersep=5pt,                  
    showspaces=false,                
    showstringspaces=false,
    showtabs=false,                  
    tabsize=2
}

\newcommand{\modulenode}[3]{
\draw(#1,#2) node[circle](#3){$A_{#3}$};
}

\newcommand{\modulenodeA}[3]{
\draw(#1,#2) node[circle](#3){$#3$};
}
\newcommand{\moduleedge}[4]{
\draw[->] (#1) -- (#2) node [midway, #4] {#3};}
\newcommand{\bk}{\ensuremath{\Bbbk}}
\newcommand{\scrm}{\ensuremath{\mathscr{M}}}
\newcommand{\m}{\ensuremath{\mathfrak{m}}}

\newcommand{\R}{\ensuremath{\mathcal{R}}}

\newcommand{\Z}{\ensuremath{\mathbb{Z}}}
\newcommand{\tangle}[1]{\ensuremath{\langle #1 \rangle}}
\newcommand{\nilhecke}[1]{\ensuremath{n\mathcal{H}(W, {\bf d}, J_{#1})}}
\newcommand{\commen}[1]{}

\DeclareMathOperator{\Rep}{\ensuremath{Rep}}

\theoremstyle{plain}
\newtheorem{theorem}[equation]{Theorem}
\newtheorem{cor}[equation]{Corollary}
\newtheorem{lemma}[equation]{Lemma}
\newtheorem{prop}[equation]{Proposition}

\newtheorem{utheorem}{\textrm{\textbf{Theorem}}}

\theoremstyle{definition}
\newtheorem{definition}[equation]{Definition}
\newtheorem{remark}[equation]{Remark}

\numberwithin{equation}{section}
\numberwithin{figure}{section}

\begin{document}
\title[The lattice of nil-Hecke algebras over real and complex reflection
groups]{The lattice of nil-Hecke algebras over\\ real and complex
reflection groups}

\author{Sutanay Bhattacharya}
\address[S.~Bhattacharya]{Indian Institute of Science, Bangalore --
560012, India; B.S.(math) Student}
\email{sutanayb@iisc.ac.in}

\author{Apoorva Khare}
\address[A.~Khare]{Department of Mathematics, Indian Institute of
Science, Bangalore, India and Analysis \& Probability Research Group,
Bangalore, India}
\email{khare@iisc.ac.in}

\date{\today}

\subjclass[2010]{20F55 (Primary), 20F05, 20C08 (Secondary)}

\keywords{Generalized Coxeter group, generalized nil-Coxeter algebra,
nil-Hecke algebra, Temperley--Lieb algebra} 

\begin{abstract}
Associated to every complex reflection group, we construct a lattice of
quotients of its braid monoid-algebra, which we term nil-Hecke algebras,
and which are obtained by killing all braid words that are ``sufficiently
long'', as well as some integer power of each generator. These include
usual nil-Coxeter algebras, nil-Temperley--Lieb algebras, and their
variants, and lead to symmetric semigroup module categories which
necessarily cannot be monoidal.

Motivated by classical work of Coxeter (1957) and the
Brou\'e--Malle--Rouquier freeness conjecture [\textit{J.\ reine Angew.\
Math.}\ 1998], and continuing beyond work of the second author
[\textit{Trans.\ Amer.\ Math.\ Soc.}\ 2018], we obtain a complete
classification of the finite-dimensional nil-Hecke algebras for all
complex reflection groups $W$. These comprise the usual nil-Coxeter
algebras for $W$ of finite type, their ``fully commutative'' analogues
for $W$ of FC-finite type, three exceptional algebras (of types $F_4,
H_3, H_4$), and three exceptional series (of types $B_n$ and $A_n$, two
of them novel). In particular, we find the first -- and only two --
finite-dimensional nil-Hecke algebras over discrete complex reflection
groups; this breaks from the nil-Coxeter case (where no braid words are
further killed, and) where Marin [\textit{J.\ Pure Appl.\ Alg.}\ 2014]
and Khare [\textit{Trans.\ Amer.\ Math.\ Soc.}\ 2018] showed that such
algebras do not exist.

In addition to these algebras, and also algebraic connections (to PBW
deformations and non-monoidal tensor categories), we further uncover
combinatorial bases of algebras, both known (fully commutative elements)
and novel ($\overline{12}$-avoiding signed permutations). Our
classification draws from and brings together results of
Popov [\textit{Comm.\ Math.\ Inst.\ Utrecht} 1982],
Stembridge [\textit{J.\ Alg.\ Combin.} 1996, 1998],
Malle [\textit{Transform.\ Groups} 1996],
Postnikov via Gowravaram--Khovanova (2015),
Hart [\textit{J.\ Group Th.}\ 2017], and
Khare [\textit{Trans.\ Amer.\ Math.\ Soc.}\ 2018].
\end{abstract}
\maketitle

\settocdepth{section}
\tableofcontents

\section{Introduction and main results}

\noindent \textit{Throughout this paper, $\bk$ will denote a fixed unital
commutative ground ring, and by $\dim M$ we mean the $\bk$-rank of a free
$\bk$-module $M$.}\medskip

In this paper, we study a class of graded rings over a given Coxeter
group $W$ (finite or infinite). We term such rings \textit{nil-Hecke
algebras}. Most of the rings that we study turn out to be the associated
graded rings for certain quotients of the corresponding (generic)
Iwahori--Hecke algebras -- which include, for instance, Temperley--Lieb
algebras. Thus, the quotients we study include the corresponding
nil-Coxeter as well as nil-Temperley--Lieb algebras.

Our focus in the present work is on exploring the finite-dimensionality
of these graded rings. In this, we are motivated by connections to
Coxeter group combinatorics, semigroup tensor categories, and flatness of
deformations (which in the complex case is related to the
Brou\'e--Malle--Rouquier (BMR) freeness conjecture). We elaborate on
these connections below.

To introduce nil-Hecke algebras, recall that any Coxeter group $W$ is
accompanied by:
\begin{itemize}
\item an associated braid group $\mathcal{B}_W \twoheadrightarrow W$;
\item a finite index set $I$ as well as a set $\mathbf{S} := \{
\mathbf{s}_i : i \in I \}$ that generates $\mathcal{B}_W$ and yields the
simple reflections under the map $\mathcal{B}_W
\twoheadrightarrow W$; and
\item a finite index set $J$ (in fact, $J \subset \binom{I}{2}$) as well
as relations $\R := \{ \mathbf{v}_j - \mathbf{w}_j : j \in J \}$, where
for each $j$, both sides denote words in $\mathbf{S}$ of equal (and
finite) length,
\end{itemize}

\noindent such that $\langle \mathbf{S} | \R \rangle$ is a presentation
of $\mathcal{B}_W$.
(This presentation was extended by Brieskorn~\cite{Bri}.)

We now define our main object of interest: a lattice of algebras
corresponding to $W$.

\begin{definition}\label{D11}
Fix a Coxeter group $W$, with associated data $I, J, {\bf S}, \R$
(so $J \subset \binom{I}{2}$).
\begin{enumerate}
\item Given $J_0 \subset J$ 
 and an integer tuple ${\bf d} = (d_i)_{i \in I}$ with $1 \leqslant d_i
\leqslant \infty\ \forall i$, define the corresponding \textit{nil-Hecke
algebra} $\nilhecke{0}$ to be the $\bk$-algebra generated by ${\bf S}$
with relations:
\[
{\bf s}_i^{d_i} = 0\ \forall i, \qquad
{\bf v}_j = {\bf w}_j\ \forall j \in J_0, \qquad
{\bf v}_j = 0 = {\bf w}_j\ \forall j \in J \setminus J_0.
\]

\item For $k \geqslant 1$, define $J_{< k} := \{ j \in J :
\ell({\bf v}_j) = \ell({\bf w}_j) < k \}$. Notice that $J_{<1} = J_{<2} =
\emptyset$, so that $\nilhecke{<1} = \nilhecke{<2}$. Also define $J_{<
\infty} := J$.
\end{enumerate}
\end{definition}

\begin{remark}
Here we omit the relation ${\bf s}_i^{d_i} = 0$ if $d_i = \infty$. Also
note that all the relations in $\R$ still hold in $\nilhecke{0}$, but the
emphasis is on the additional relations ${\bf v}_j = 0 = {\bf w}_j$, $j
\in J \setminus J_0$.
\end{remark}

In other words, $\nilhecke{0}$ is a quotient of the Artin monoid-algebra
of $W$ by the order relations ${\bf s}_i^{d_i} = 0$ (see \cite{Pa}). Two
special cases of this construction are:
\begin{itemize}
\item ($k=\infty$.)
The \textit{generalized nil-Coxeter algebra} associated to the above data
is $NC_W({\bf d}) := \nilhecke{< \infty}$.
These algebras were defined and studied in previous work \cite{Kh,Kh2}
(see also the extended abstract \cite{Kh-FPSAC}).
In the special case where 
 $d_i = 2\ \forall i$, $NC_W((2,\dots,2))$ is the ``usual'' nil-Coxeter
algebra associated to $W$, with $J = \binom{I}{2}$. These algebras are
well-studied in connection to flag varieties \cite{BGG,KK},
categorification \cite{Kho,KL}, symmetric function theory \cite{BSS}, and
Schubert polynomials \cite{FS,LS}.

\item ($k=3$.)
We define the \textit{generalized nil-Temperley--Lieb algebra}
associated to the above data to be $NTL_W({\bf d}) := \nilhecke{< 3}$. In
the special case where 
 $d_i = 2\ \forall i$, $NTL_W((2,\dots,2))$ is the ``usual''
nil-Temperley--Lieb algebra associated to $W$. These latter are quotients
of nil-Coxeter algebras as well as associated graded algebras of
Temperley--Lieb algebras (see e.g.~\cite{BM,FG}), and have connections to
fermionic particle configurations and to quantum Schubert calculus
\cite{KS,Pos}. 
 They were defined and studied in \cite{Fan1,Fan2,Gr}. In the 
 special case of the Coxeter (multi)graph of $W$ being simply laced, the
algebras $NTL_W((2,\dots,2))$ were rediscovered by Postnikov as
\textit{XYX algebras} \cite{GK}.
\end{itemize}

For a given group $W$ with presentation as above, the set of nil-Hecke
$\bk$-algebras forms a lattice isomorphic to the product of the power set
$2^J$ with $I$ copies of $\Z^{\geqslant 1} \sqcup \{ \infty \}$. Two
algebras in this lattice are comparable if and only if one surjects onto
the other; the extremal points in the lattice are the source
$\mathcal{B}_W$ (for $J_0 = J$ and $d_i = \infty\ \forall i$) and the
rank-$1$ algebra $\bk$ (whenever $d_i = 1\ \forall i$). For each ${\bf d}
\in (\Z^{\geqslant 1})^I$, the corresponding algebras form a sub-lattice
with unique sink $n\mathcal{H}(W,{\bf d},\emptyset)$ and source the
generalized nil-Coxeter algebra $NC_W({\bf d})$.

We now come to the results in this paper, which classify all nil-Hecke
$\bk$-algebras $\nilhecke{<k}$ for $1 \leqslant k < \infty$ that are
finite-dimensional (for $\bk$ a field, or more precisely, those of finite
$\bk$-rank) of the form $\nilhecke{<k}$ with $1 \leqslant k \leqslant
\infty$. In this, we are motivated by classical work of Coxeter
\cite{Co}, in which he considered quotients of the type $A$ braid group
by the relations $s_i^p = 1\ \forall i \geqslant 1$, and classified the
pairs $(n,p)$ for which the resulting ``generalized Coxeter group'' is
finite (see also \cite{As}). Our work studies the ``Iwahori--Hecke''
analogue of Coxeter's problem -- and in the complex groups case, it
can be thought of as related to the ``Temperley--Lieb'' analogue of the
Brou\'e--Malle--Rouquier freeness conjecture \cite{BMR1,BMR2}.

At the same time, our work also complements (by proving results for all
finite $k$) the work of the second author \cite{Kh2}, who studied the $k
= \infty$ case. Furthermore, we unify under one umbrella several
classical and recent works, including (\cite{Kh2}, as well as) results by
Stembridge \cite{St1,St3,St2}, Gowravaram--Khovanova (who attribute their
construction to Postnikov) \cite{GK}, and Hart \cite{Ha}, and show how
they fit into this picture.

\subsection{Classification of finite-dimensional cases}

We begin with the recent work \cite{Kh2}, which classified all
finite-dimensional objects among the generalized nil-Coxeter algebras
(i.e., corresponding to $k = \infty$). Interestingly, this yields a
``non-usual'' type $A$ family of such algebras -- and it turns out to be
the only one:

\begin{theorem}[\cite{Kh2}]\label{Tinfty}
Fix a Coxeter group $W$ with related
data $I,J,{\bf S},\R$, and integers $d_i \geqslant 2\ \forall i$. The
corresponding nil-Hecke (or generalized nil-Coxeter) algebra
$\nilhecke{<\infty}$ is a finitely generated $\bk$-module if and only if
exactly one of the following occurs:
\begin{enumerate}
\item $W$ is a finite Coxeter group and ${\bf d} = (2,\dots,2)$ (the
``usual'' nil-Coxeter algebras); or
\item $W$ is of type $A$,
and ${\bf d} = (d,2,\dots,2)$ or $(2,\dots,2,d)$ for some $d>2$.
We denote this nil-Hecke algebra by $NC_A(n,d) :=
n\mathcal{H}(S_{n+1}, (d,2,\dots,2), J_{<\infty})$.
\end{enumerate}
\end{theorem}

\begin{remark}
Note that setting any $d_i = 1$ yields degeneracies in the original
Definition~\ref{D11}, since the generator ${\bf s}_i$ is rendered
superfluous. Moreover,  if $d_i = \infty$ then $\nilhecke{0}$ contains
the $\bk$-span of ${\bf s_i}^n$ for $n \geqslant 0$. As our goal in this
paper is to classify the finite-rank nil-Hecke algebras, throughout this
paper we assume $2 \leqslant d_i < \infty\ \forall i \in I$.
\end{remark}

Following Theorem~\ref{Tinfty} from \cite{Kh2}, the present work obtains
the corresponding classification results for $J_{<k}$ with $k < \infty$.
Our first result is for $k = 1,2$:

\begin{utheorem}[$k=1,2$]\label{T12}
Fix a Coxeter group $W$ with related data $I,J,{\bf S},\R$, and integers
$d_i \geqslant 2\ \forall i$. The corresponding nil-Hecke $\bk$-algebra
$\nilhecke{<1} = \nilhecke{<2}$ is a finitely generated $\bk$-module if
and only if exactly one of the following occurs:
\begin{enumerate}
\item 
 The Coxeter (multi)graph of $W$ is a tree with exactly one multiple
edge $4 \leqslant m_{i_0 j_0} < \infty$, and $d_i = 2\ \forall i$. In
this case, if disconnecting $i_0, j_0$ yields trees $T, T'$ with $a,b$
nodes respectively, then
\[
\dim \nilhecke{<2} = \begin{cases}
\frac{1}{2} m_{i_0 j_0} |I|^2 + 1 - 2ab, \qquad & \text{if $m_{i_0 j_0}$
is even},\\
\frac{1}{2} (m_{i_0 j_0} - 1) |I|^2 + 1, & \text{if $m_{i_0 j_0}$ is
odd}.
\end{cases}
\]

\item 
 The Coxeter (multi)graph of $W$ is a simply laced tree, and at most one
node $i_0 \in I$ satisfies: $d_{i_0} \geqslant 3$. In this case,
\[
\dim \nilhecke{<2} = 1 + |I|^2 (d_{i_0} - 1).
\]
\end{enumerate}
\end{utheorem}

\begin{remark}
Here and below, by the statement ``$\dim \nilhecke{0} = m$'' for an
integer $m>0$, we mean that $\nilhecke{0}$ is a free $\bk$-module of rank
$m$.
\end{remark}

Thus for $k=1,2$, there are infinitely many finite-dimensional nil-Hecke
algebras corresponding to ``non-Coxeter'' matrices (i.e., where $d_{i_0}
> 2$). In contrast, our next result shows that for $k \geqslant 3$, this
already gets reduced to precisely one non-Coxeter family, which for $k =
\infty$ corresponds to $NC_A(n,d)$ in Theorem \ref{Tinfty}(2).

\begin{utheorem}[$k=3$]\label{T3}
Fix a Coxeter group $W$ with related
data $I,J,{\bf S},\R$, and integers $d_i \geqslant 2\ \forall i$. The
corresponding nil-Hecke (or generalized nil-Temperley--Lieb)
$\bk$-algebra $\nilhecke{<3} = NTL_W({\bf d})$ is a finitely generated
$\bk$-module if and only if exactly one of the following occurs:
\begin{enumerate}
\item $W$ is a finite Coxeter group and ${\bf d} = (2,\dots,2)$ (these
are ``usual'' nil-Temperley--Lieb algebras);

\item $W$ is a Coxeter group of type $E_n (n \geqslant 9), \ F_n (n
\geqslant 5)$, or $H_n (n \geqslant 5)$ (see Figure \ref{Fig1}), and
${\bf d} = (2,\dots,2)$ (these are also ``usual'' nil-Temperley--Lieb
algebras); or

\item $W$ is of type $A$,
and ${\bf d} = (d,2,\dots,2)$ or $(2,\dots,2,d)$ for some $d>2$.
\end{enumerate}

In the first two cases, $\dim NTL_W({\bf d})$ is precisely the number of
\textit{fully commutative} elements in the corresponding Coxeter group
$W$, i.e., the number of words $w \in W$ for which any reduced word can
be obtained from any other without using the ``non-commutative'' braid
relations ($m_{ij}\geqslant 3$). In the last case of $W = W(A_n)$ and
$d_1 > 2 = d_2 = \cdots = d_n$, we have:
\begin{equation}\label{Exyx}
\dim NTL_W({\bf d}) = (d-1)C_{n+1}-(d-2)C_n+(d-2)\sum_{j=1}^{n-1}jC_{n-j}
\end{equation}
where $C_n$ is the $n^{\text{th}}$ Catalan number.
\end{utheorem}

\begin{figure}[ht]
\begin{tikzpicture}[line cap=round,line join=round,>=triangle 45,x=1.0cm,y=1.0cm]
\draw(1,12) circle (0.2cm);
\draw(2,12) circle (0.2cm);
\draw(3,12) circle (0.2cm);
\draw(4,12) circle (0.2cm);
\draw(7,12) circle (0.2cm);
\draw(3,13) circle (0.2cm);
\draw (0.75,12.25) node[anchor=north west] {{\small 1}};
\draw (1.75,12.25) node[anchor=north west] {{\small 2}};
\draw (2.75,12.25) node[anchor=north west] {{\small 4}};
\draw (3.75,12.25) node[anchor=north west] {{\small 5}};
\draw (5.25,12.2) node[anchor=north west] {$\cdots$};
\draw (6.75,12.2) node[anchor=north west] {\textit{n}};
\draw (2.75,13.25) node[anchor=north west] {{\small 3}};
\draw (9,12.25) node[anchor=north west] {{$E_n \ (n \geqslant 9)$}};
\draw [-] (1.2,12) -- (1.8,12);
\draw [-] (2.2,12) -- (2.8,12);
\draw [-] (3.2,12) -- (3.8,12);
\draw [-] (4.2,12) -- (4.8,12);
\draw [-] (6.2,12) -- (6.8,12);
\draw [-] (3,12.2) -- (3,12.8);
\draw(1,11) circle (0.2cm);
\draw(2,11) circle (0.2cm);
\draw(3,11) circle (0.2cm);
\draw(4,11) circle (0.2cm);
\draw(7,11) circle (0.2cm);
\draw (0.75,11.25) node[anchor=north west] {{\small 1}};
\draw (1.75,11.25) node[anchor=north west] {{\small 2}};
\draw (2.75,11.25) node[anchor=north west] {{\small 3}};
\draw (3.75,11.25) node[anchor=north west] {{\small 4}};
\draw (5.25,11.2) node[anchor=north west] {$\cdots$};
\draw (6.75,11.2) node[anchor=north west] {\textit{n}};
\draw (9,11.25) node[anchor=north west] {{$F_n \ (n \geqslant 5)$}};
\draw [-] (1.2,11) -- (1.8,11);
5\draw [-] (2.2,10.95) -- (2.8,10.95);
\draw [-] (2.2,11.05) -- (2.8,11.05);
\draw [-] (3.2,11) -- (3.8,11);
\draw [-] (4.2,11) -- (4.8,11);
\draw [-] (6.2,11) -- (6.8,11);
\draw(1,10) circle (0.2cm);
\draw(2,10) circle (0.2cm);
\draw(3,10) circle (0.2cm);
\draw(4,10) circle (0.2cm);
\draw(7,10) circle (0.2cm);
\draw (0.75,10.25) node[anchor=north west] {{\small 1}};
\draw (1.75,10.25) node[anchor=north west] {{\small 2}};
\draw (2.75,10.25) node[anchor=north west] {{\small 3}};
\draw (3.75,10.25) node[anchor=north west] {{\small 4}};
\draw (5.25,10.2) node[anchor=north west] {$\cdots$};
\draw (6.75,10.2) node[anchor=north west] {\textit{n}};
\draw (9,10.25) node[anchor=north west] {{$H_n \ (n \geqslant 5)$}};
\draw [-] (1.2,10) -- (1.8,10);
\draw [-] (1.2,9.9) -- (1.8,9.9);
\draw [-] (1.2,10.1) -- (1.8,10.1);
\draw [-] (2.2,10) -- (2.8,10);
\draw [-] (3.2,10) -- (3.8,10);
\draw [-] (4.2,10) -- (4.8,10);
\draw [-] (6.2,10) -- (6.8,10);
\end{tikzpicture}
\caption{Dynkin graphs with $|W|$ infinite but $\dim(NTL_W)$
finite}\label{Fig1}
\end{figure}
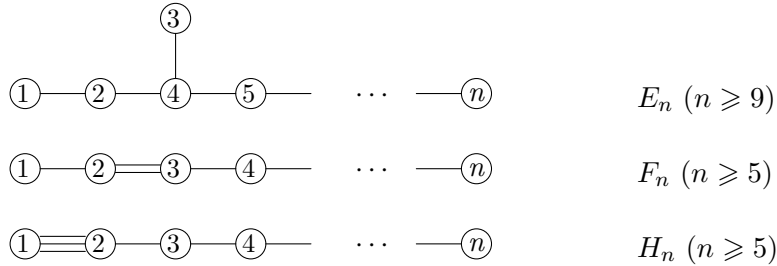

\begin{remark}
Following Postnikov's calling $NTL_{W(A_n)}((2,\dots,2))$ an
\textit{XYX-algebra} (see \cite{GK}), we term
$NTL_{W(A_n)}((d,2,\dots,2)) = n\mathcal{H}(W(A_n), (d,2,\dots,2),
J_{<3})$ a \textit{generalized XYX-algebra}. Its $\bk$-rank \eqref{Exyx}
involving Catalan numbers, generalizes the $d=2$ fact that the type $A$
XYX-algebra has rank $C_{n+1}$ -- a fact first observed in \cite{BJS}.
We also remark that the final sum in~\eqref{Exyx} is the partial
sum of the partial sum of the Catalan numbers, and its first few terms
can be found in \cite{OEIS}.
\end{remark}

In particular, Theorem \ref{T3} together with \cite[Theorem B]{Kh2}
implies Theorem \ref{Tinfty} for all generalized Coxeter matrices. This
because all Coxeter group algebras $\bk W$ are flat deformations of the
corresponding nil-Coxeter algebras $NC_W((2,\dots,2))$.

As the algebras $NTL_W({\bf d})$ are not our primary focus in the present
paper, we will not go beyond the proof of Theorem \ref{T3}; however, we
remark that the proof of Theorem \ref{T3} is constructive, and reveals
the analogue of fully commutative elements in the ``non-usual'' case of
the generalized XYX-algebras $NTL_{A_n}((2,\dots,2,d)), \ d > 2$. Recall
that such elements were introduced and studied by Stembridge in
\cite{St1,St3}, and enumerated in his sequel
\cite{St2}.
In \cite{St1}, Stembridge also computed the Coxeter groups which contain
only finitely many fully commutative elements (see also \cite{Fan0,Gr}):
these are the finite Coxeter groups as well as the families $E_n, F_n,
H_n$ as above.
Theorem~\ref{T3} now constructs and enumerates such elements in the
``non-usual'' generalized nil-Coxeter algebras $NC_A(n,d)$ (see
Theorem~\ref{Tinfty}(2)) introduced and studied in \cite{Kh2}, adding to
our knowledge of them.\medskip


Theorem~\ref{T3} classified the finite $\bk$-rank nil-Hecke algebras for
$k=3$. Next, we look at the case $k=4$, followed by an immediate
corollary for $k=5$. 

\begin{utheorem}[$k=4$]\label{T4}
Fix a Coxeter group $W$ with related data $I,J,{\bf S},\R$, and integers
$d_i \geqslant 2\ \forall i$. The corresponding nil-Hecke $\bk$-algebra
$\nilhecke{<4}$ is a finitely generated $\bk$-module if and only if
exactly one of the following occurs:
\begin{enumerate}
    \item $W$ is a finite Coxeter group with ${\bf d}=(2,\ldots,2)$; or
    \item $W$ is of type $A$,
and ${\bf d} = (d,2,\dots,2)$ or $(2,\dots,2,d)$ for some $d>2$.
\end{enumerate}
In the first case, $\dim\nilhecke{<4}$ is given by the following table:
\[
\begin{array}{c|c}
\hline
     W & \dim\nilhecke{<4}\\
    \hline
    A_n, D_n, E_6,E_7, E_8 & |W|\\
    
    B_n & \sum_{k=0}^n\binom{n}{k}^2k!\\
   
    F_4 & 304\\
    H_3 & 76\\
    H_4 & 1460\\
   
    I_2(m) & \begin{cases}
         2m & \text{ if } m<4  \\
          2m-1& \text{ if } m \geqslant 4
    \end{cases}\\
    \hline
\end{array}
\]
In the second case, we have $$\dim\nilhecke{<4}=n!(1+n(d-1)).$$
\end{utheorem}

\begin{cor}[$k=5$]\label{T5}
Fix a Coxeter group $W$ with related data $I,J,{\bf S},\R$, and integers
$d_i \geqslant 2\ \forall i$. The corresponding nil-Hecke $\bk$-algebra
$\nilhecke{<5} $ is a finitely generated $\bk$-module if and only if
exactly one of the following occurs:
\begin{enumerate}
    \item $W$ is a finite Coxeter group with ${\bf d}=(2,\ldots,2)$; or
    \item $W$ is of type $A$,
and ${\bf d} = (d,2,\dots,2)$ or $(2,\dots,2,d)$ for some $d>2$.
\end{enumerate}
In the first case, $\dim\nilhecke{<5}$ is given by the following table:

\[
\begin{array}{c|c}
\hline
     W & \dim\nilhecke{<5}\\
    \hline
    A_n, B_n, D_n, E_6,E_7, E_8, F_4 & |W|\\
    
    H_3 & 76\\
    H_4 & 1460\\
   
    I_2(m) & \begin{cases}
         2m & \text{ if } m<5  \\
          2m-1& \text{ if } m \geqslant 5
    \end{cases}\\
    \hline
\end{array}
\]
In the second case, we have $$\dim\nilhecke{<5}=n!(1+n(d-1)).$$
\end{cor}

\begin{remark}
In particular, this work provides some refined combinatorial information
that (to the best of our knowledge) was not written down for ``usual''
nil-Coxeter algebras -- nor for
their ``non-usual'' (aka generalized, i.e.\ with ${\bf d} \neq
(2,\dots,2)$) counterparts in type $A$.
\begin{itemize}
\item For ``usual'' nil-Coxeter algebras, we prove the above results by
computing not just the dimension (or $\bk$-rank), but by isolating -- via
Theorem~\ref{Tdim} -- a distinguished subset of words $S'_k$ inside the
``usual'' monomial basis, which provides a $\bk$-basis. For $k=1,2$, the
set $S'_k$ is simply the full basis (of size $|W|$), and for $k=3$ it
corresponds to the subset $W_{\rm fc}$ of fully commutative elements.
However, for $k=4,5$, $S'_k$ was not isolated earlier -- even for types
$B_n, F_4, H_3, H_4$ if we restrict to $S'_k$ finite.

There is also combinatorial information, which to the best of our
knowledge was not previously connected to algebra. Namely, in type $B_n$
in Theorem~\ref{T4}, the monomials in $S'_4$ that are a basis of the
algebra $\nilhecke{<4}$ -- i.e., which avoid braid relations that are
``long enough'' -- turn out to be precisely the
``$\overline{1}\overline{2}$-avoiding signed permutations''.

\item Moving from combinatorics to algebra: in addition to our
identifying a subset of words with some combinatorial properties -- our
work moreover isolates a multiplication structure on their span. This is
akin to the set of fully commutative words inside $W$ -- while their
closure under multiplication in $W$ is unclear, this indeed happens when
looking at their images inside the corresponding nil-Temperley--Lieb
algebras. Similarly, the construction of ``usual'' nil-Hecke algebras in
this work reveals a multiplication structure on the $\bk$-span of the set
$S'_k$ for $k=4,5$, which is not revealed when considering $S'_k$ as a
subset of the Coxeter group $W$.

\item In parallel, in the ``non-usual'' case, the work \cite{Kh2} that
introduced the algebras $NC_A(n,d)$ (see Theorem~\ref{Tinfty}(2)) could
not detect the analogue of the fully commutative elements, which we now
uncover (via Theorem~\ref{Tdimplus}) in the proof of Theorem~\ref{T3}.
\end{itemize}
\end{remark}

Returning to the classification of finite-rank nil-Hecke algebras, our
final result is for $k \geqslant 6$.

\begin{cor}[$6 \leqslant k \leqslant \infty$]\label{T6}
Fix a Coxeter group $W$ with related data $I,J,{\bf S},\R$, integers $d_i
\geqslant 2\ \forall i$ and $6 \leqslant k \leqslant \infty$. The
corresponding nil-Hecke $\bk$-algebra $\nilhecke{<k} $ is a finitely
generated $\bk$-module if and only if $W$ is either a finite Coxeter
group, with ${\bf d} = (2,\ldots, 2)$ or $W$ is of type $A$ with ${\bf d}
= (d,2,\ldots,2)$ or $(2,\ldots,2,d)$ for some $d>2$. In the first case,
the corresponding nil-Hecke algebra has $\bk$-rank equal to $|W|$ unless
it is of the type $I_2(m)$, in which case it is given by
\[
\dim\nilhecke{<k} = \begin{cases}
2m &\text{if } m<k,\\
2m-1 & \text{if } m \geqslant k.\end{cases}
\]
In the second case, the $\bk$-rank is $n!(1+n(d-1))$.
\end{cor}

For ease of reference, we tabulate the above results for $3 \leqslant k
\leqslant \infty$ in Table~\ref{Ttable}, and follow it up with a few
comments.

\begin{table}[ht]
\hspace*{-8mm}\begin{tabular}{| c || c | c | c | c |}
\hline
$W$ & $J_{<3}$ & $J_{<4}$ & $J_{<5}$ & $J_{<k}$ ($6 \leqslant k \leqslant
\infty$)\\
\hline
\hline
$A_n, D_n, E_6, E_7, E_8$  & $|W_{\rm fc}|$ & $|W|$  & $|W|$     & $|W|$\\
$B_n$                      & $|W_{\rm fc}|$ & $\displaystyle \sum_{k=0}^n
\binom{n}{k}^2 k!$ & $|W|$ & $|W|$\\
$F_4$                      & $|W_{\rm fc}|$ & $304$  & $|W|$     & $|W|$\\
$H_3$                      & $|W_{\rm fc}|$ & $76$   & $76$      & $|W|$\\
$H_4$                      & $|W_{\rm fc}|$ & $1460$ & $1460$    & $|W|$\\
\hline
$I_2(m)$                   & $|W|$ if $m<3$, &
$|W|$ if $m<4$, & $|W|$ if $m<5$,  & $|W|$ if $m<k$\\
& else $|W_{\rm fc}| = |W|-1$ & else $|W|-1$ & else $|W|-1$
& else $|W|-1$\\
\hline
$E_n$ $(n \geqslant 9)$, & & & & \\
$F_n$ $(n \geqslant 5)$, & $|W_{\rm fc}|$ & & & \\
$H_n$ $(n \geqslant 5)$ & & & & \\
\hline
$A_n$, ${\bf d} = (d,2,\dots,2)$, & (see \eqref{Exyx})
& $n! (1+n(d-1))$ & $n! (1+n(d-1))$ & $n! (1+n(d-1))$\\
$d>2$ & & & & \\
\hline
\end{tabular}\bigskip
\caption{Classification of Coxeter groups with finite-dimensional
nil-Hecke algebras $\nilhecke{<k}$ for $3 \leqslant k \leqslant \infty$.
Except the last row, ${\bf d} = (2,\dots,2)$. $W_{\rm fc}$ denotes
the fully commutative elements in $W$. Each cell entry is the
dimension/rank over $\bk$.}\label{Ttable}
\end{table}
\bigskip

\begin{remark}
The fully commutative words in all Coxeter groups were studied and
enumerated by Stembridge in \cite{St2}. In particular, for $W$ dihedral
(i.e., of type $I_2(m)$), $W_{\rm fc} = W \setminus \{ w_\circ \}$
comprises all non-longest elements. Thus in Table~\ref{Ttable}, the
$\bk$-rank of $\nilhecke{<k}$ over $W = W(I_2(m))$ with $m \geqslant 3$
is $|W|$ if $m < k$, and $|W| - 1 = |W_{\rm fc}|$ otherwise.

This shows that the finite rank nil-Hecke algebras have ranks either
(a)~equal to $|W|$ or $|W_{\rm fc}|$, or
(b)~belonging to the three exceptional cases $F_4, H_3, H_4$ with $k=4$,
or
(c)~belonging to the three exceptional families: $B_n$ with $k=4$, or
$A_n$ with $k=3,4$ and ${\bf d} = (2,\dots,2,d)$ for some $d \geqslant
2$.

Finally, note that as $\nilhecke{<k}$ is a quotient of $\nilhecke{<l}$
for $1 \leqslant k \leqslant l \leqslant \infty$, the list of
finite rank examples for any value of $k$ -- i.e., in any non-initial
column in Table~\ref{Ttable} -- is contained in the rows occupied in the
previous column. (This is useful in the proofs below.) In particular, the
result in \cite{Kh2} that there is only one family of finite rank
generalized nil-Coxeter algebras with ${\bf d} \neq (2,\dots,2)$, already
follows from the $J_{<4}$ column in Table~\ref{Ttable}.
\end{remark}

Our next result provides an equivalent condition to the
finite-dimensionality of nil-Hecke algebras, motivated by considerations
of deformation theory (see the end of Section~\ref{S2}):

\begin{utheorem}\label{TNil}
Fix a Coxeter group $W$ with related data $I,J,{\bf S},\R$, and integers
$d_i \geqslant 2\ \forall i$, and let $k$ be an integer such that $1
\leqslant k \leqslant \infty$. Then the following are equivalent:
\begin{enumerate}
\item The corresponding nil-Hecke $\bk$-algebra $\nilhecke{<k}$ is a
finitely generated $\bk$-module.

\item The nil-Hecke algebra is found in the lists in Theorems \ref{T12},
\ref{T3}, \ref{T4} and Corollaries \ref{T5}, \ref{T6} -- or equivalently,
in Theorem~\ref{T12} and Table~\ref{Ttable}.

\item The two-sided augmentation ideal $\m \subset \nilhecke{<k}$, which
is generated by $\{ {\bf s}_i : i \in I \}$, is nilpotent.
\end{enumerate}
If $\bk$ is moreover a field, then the finite-dimensional $\bk$-algebra
$\nilhecke{<k}$ is local, with unique maximal ideal $\m$.
\end{utheorem}

Note that Theorem~\ref{TNil} generalizes the assertion \cite[Theorem
D(3)]{Kh2}, which is simply the $k=\infty$ special case.

\subsection{Complex reflection groups; Frobenius algebras}

We next explore the analogous picture when working with nil-Hecke
$\bk$-algebras over complex reflection groups. Recall that such groups
also have Coxeter-type presentations, listed e.g.\ in
\cite{BMR2,Mal,Po1}. In fact this holds not only for the finite complex
reflection groups (e.g.\ classified by Sheppard--Todd \cite{ST}), but
also for the infinite discrete groups generated by affine unitary
reflections -- these were classified by Popov \cite{Po1}. In the sequel,
we term these (finite or infinite) groups as \textit{discrete complex
reflection groups}.

Given such a complex reflection group -- which we continue to denote by
$W$ -- one can define a lattice of nil-Hecke type algebras
$\nilhecke{<k}$ over $\bk$, and we define these following \cite{Be,BMR2}
(and as in \cite{Kh2}):

\begin{definition}\label{Dfinite}
Suppose $W$ is a discrete (finite or infinite) complex reflection group,
together with a finite generating set of complex reflections $\{ s_i : i
\in I \}$, the order relations $s_i^{m_{ii}} = 1\ \forall i$, a set of
braid relations $\{ R_j : j \in J \}$ -- each involving words with at
least two distinct reflections $s_i$ -- and for the infinite non-Coxeter
complex reflection groups $W$ listed in \cite[Tables I, II]{Mal}, one
more order relation $R_0^{m_0} = 1$.
Now define $I_0 := I \sqcup \{ 0 \}$ for these infinite non-Coxeter
complex reflection groups $W$, and $I_0 := I$ otherwise. Given an integer
vector ${\bf d} \in \mathbb{N}^{I_0}$ with $d_i \geqslant 2\ \forall i$,
define the corresponding \textit{nil-Hecke algebra} to be
\begin{equation}
\nilhecke{<k} :=
\frac{\bk \tangle{\mathbf{s}_i, i \in I}} {( \{ R_j' = 0, j \in J_{<k} \},
\mathbf{s}_i^{d_i} = 0\ \forall i \in  I, \ (R'_0)^{d_0} = 0)},
\end{equation}

\noindent where the braid relations $R_j$ are replaced by the
corresponding relations $R'_j$ in the alphabet $\{ \mathbf{s}_i : i \in I
\}$, and similarly for $R'_0$ if $R_0^{m_0} = 1$ in $W$; and where the
two braid words on either side of the relations $R'_j$ that are of equal
length $\geq k$ (i.e., $j \not\in J_{<k}$) are both set to equal zero.
There is also the notion of the corresponding \textit{braid diagram} as
in \cite[Tables 1--4]{BMR2} and \cite[Tables I, II]{Mal}; this is no
longer always a Coxeter graph.
\end{definition}

\begin{remark}
As Popov explains in \cite[Section 1.6]{Po1}, one needs to work in the
preceding definition with a specific presentation for $W$, since there is
no canonical (minimal) set of generating reflections. Some related
results are found in \cite{Bas1}.
\end{remark}

We now present our next main theorem. Parallel to the above results, it
is natural to ask which of these algebras have finite $\bk$-rank, even if
(parallel to the BMR freeness conjecture) this never happens for
$k=\infty$. For $k = \infty$, this was explored in previous work
\cite{Kh2}, where it was shown that no such group yields finite-rank
Hecke algebras. This result subsumes an interesting observation of Marin
\cite{Ma}, where he says that a key difference between real and complex
reflection groups $W$ is that there are no nil-Coxeter $|W|$- (or
finite-)dimensional algebras over complex reflection groups.

Our next theorem explores the (larger) lattice of nil-Hecke algebras --
and shows that Marin's observation, as well as the result in \cite{Kh2},
hold more generally for all $k>2$ -- for \textit{finite} complex
reflection groups. However, in a break with these results, we provide the
first -- and only! -- two examples of finite-dimensional nil-Hecke
algebras, and they arise over infinite discrete complex reflection
groups:

\begin{utheorem}\label{Tcomplex}
Suppose $3 \leq k \leq \infty$, and $W$ is any irreducible discrete
complex reflection group, i.e., $W$ is a complex reflection group with
connected braid diagram and presentation given as in
\cite[Tables 1--4]{BMR2},
\cite[Tables I, II]{Mal}, or
\cite[Table 2]{Po1}.
Also fix an integer vector ${\bf d}$ with $d_i \geqslant 2\ \forall i$
(including possibly for the additional order relation as in \cite{Mal}).
\begin{enumerate}
\item If $W$ is finite or genuine crystallographic (this notation is
explained in the proof), then $\nilhecke{<k}$ has infinite $\bk$-rank.

\item If $W$ is noncrystallographic then either $\nilhecke{<k}$ has
infinite $\bk$-rank, or $k=3$ and $W$ is the complexification of an
affine Coxeter group of type $E_9$ or $F_5$.
\end{enumerate}
\end{utheorem}

In particular, the $k=4$ case of this result implies the $5 \leq k <
\infty$ cases -- as well as the $k=\infty$ case, which was one of the
main results in previous work \cite{Kh2}.

\begin{remark}
We also discuss the remaining case in \cite{Po1} of irreducible discrete
complex reflection groups -- wherein $W$ is ``non-genuine''
crystallographic. As we explain using recent results in \cite{IS}, the
Coxeter-type presentation that we use (from \cite{IS}) involves relations
whose left and right sides have unequal lengths, and so we do not
consider these algebras any further.
\end{remark}

Our final main result discusses which of the finite-dimensional
$\bk$-algebras are Frobenius. (For this result we assume that $\bk$ is a
field.) This line of investigation was first explored by Khovanov, who
pointed out \cite{Kho} that the ``usual'' nil-Coxeter algebra over any
finite Coxeter group $W$ is always Frobenius. To proceed further, we need
recall additional notation.

\begin{definition}
Given a nil-Hecke algebra $\nilhecke{<k}$, we say that an element
$x\in\nilhecke{<k}$ is \textit{left-primitive} if $\mathfrak m x=0$, and
\textit{right-primitive} if $x\mathfrak m=0$, where $\mathfrak m$ is the
ideal defined in Theorem \ref{TNil}. We call $x$ \textit{primitive} if it
is both right-primitive and left-primitive.
\end{definition}

Now our final main result characterizes the Frobenius algebras among all
finite-dimensional nil-Hecke algebras: (i)~in terms of their primitive
elements; and (ii)~for all $k \geq 1$. The latter means that our result
again extends the classification in the $k=\infty$ case in \cite{Kh2}.
In light of Theorem~\ref{Tcomplex}, we only consider Coxeter groups.

\begin{utheorem}\label{TFrobenius}
Fix a field $\bk$, a Coxeter group $W$ with related data $I,J,{\bf
S},\R$, integers $d_i \geqslant 2\ \forall i$, and $1 \leqslant k
\leqslant \infty$, such that the corresponding nil-Hecke algebra
$\nilhecke{<k}$ is finite-dimensional. Then the following are equivalent:
\begin{enumerate}
\item The algebra $\nilhecke{<k}$ is Frobenius.
\item The space of right-primitive (equivalently, left-primitive)
elements is one-dimensional.
\item The spaces of right-primitive, left-primitive, and primitive
elements are all the same and have dimension one.
\item Exactly one of the following holds:
(a)~$W$ is a Coxeter group with no braid relation of length bigger than
$k$, and $d_i=2$ for all $i$; or
(b)~$W$ is of type $A_1$, with $\mathbf d=(d)$ for some $d\geqslant 3$.
\end{enumerate}
\end{utheorem}

Once again, the $k=3$ case subsumes ``one implication'' of the $k=\infty$
case shown originally in \cite{Kh2}. Moreover, Theorem~\ref{TFrobenius}
shows that the nil-Coxeter algebras occupy a distinguished position among
the nil-Temperley--Lieb and other nil-Hecke algebras (even when all $d_i
= 2$): the others are not Frobenius.

\subsection*{Organization}

This paper is organized as follows.
We begin by providing several motivations and connections to prior works
in the following section, as well as by producing a monomial word
$\bk$-basis for $\nilhecke{<k}$ for general $W$, ${\bf d} \in
(\Z^{\geqslant 2})^I$, and $1 \leqslant k \leqslant \infty$. The
remaining sections are devoted to proving the main theorems above.
In an Appendix, we provide SAGE code that helped verify some of our
results computationally.

\subsection*{Acknowledgments}
The second author was partially supported by Ramanujan Fellowship grant
SB/S2/RJN-121/2017, MATRICS grant MTR/2017/000295, and SwarnaJayanti
Fellowship grants SB/SJF/2019-20/14 and DST/SJF/MS/2019/3 from SERB and
DST (Govt.\ of India).

\section{Flat deformations, other connections, and a word basis}\label{S2}

In this section we explain how the problem under consideration, and more
broadly, nil-Hecke algebras, are connected to several other areas. We
begin with a categorical connection: even before considering
finite-dimensionality, the entire lattice of nil-Hecke algebras $\{
\nilhecke{0} : {\bf d} \in (\Z^{\geqslant 2})^I, \ J_0 \subset J \}$
yields examples of co-commutative algebras with coproduct, which are not
bialgebras.

\begin{prop}\label{Pnotbialg}
Let $A := \nilhecke{0}$ for some ${\bf d} \in (\Z^{\geqslant 2})^I$ and
$J_0 \subset J$. If $\bk$ is a field of characteristic zero, then $A$ is
not a bialgebra.
\end{prop}

\begin{proof}
We provide a sketch, as the proof is similar to that of \cite[Proposition
3.1]{Kh2}. Since each $d_i < \infty$, the only possible counit
$\varepsilon : A \to \bk$ sends each ${\bf s}_i$ to $0$. Let
$\m$ denote the two-sided augmentation ideal in $A$, generated by $\{
{\bf s}_i : i \in I \}$; then one shows that any coproduct $\Delta : A
\to A \otimes A$ must satisfy:
\[
\Delta({\bf s}_i) \in 1 \otimes {\bf s}_i + {\bf s}_i \otimes 1 + \m
\otimes \m.
\]
But now,
\[
0 = \Delta({\bf s}_i)^{d_i} = \sum_{k=1}^{d_i - 1} \binom{d_i}{k} {\bf
s}_i^k \otimes {\bf s}_i^{d_i - k} + \text{ higher degree terms},
\]
and this is impossible because $\nilhecke{0}$ surjects onto the quotient
of the usual nil-Coxeter algebra $NC_W((2,\dots,2))$ by some braid words
(each of length at least $2$), so the ${\bf s}_i$ have nonzero images in
this quotient.
\end{proof}

A consequence of this result is that the module category $\Rep A$ (with
$A = \nilhecke{0}$) cannot be a tensor category. But since $\Delta : A
\to A \otimes A$ is co-commutative, $\Rep A$ is a symmetric semigroup
category. We refer the reader to \cite[Theorem A]{Kh2} and
\cite[Proposition 14.2 and Theorem 18.3]{ES} for more on the Tannakian
formalism behind such categories. Additionally, we mention for
completeness how to pass to an ``honest'' tensor category from $A$.
Define a central extension of the graded algebra $A$ by an element ${\bf
s}_\infty$:
\[
0 \to \bk {\bf s}_\infty \to \widetilde{A} \to A \to 0,
\]
with the non-graded $\bk$-algebra $\widetilde{A}$ generated by ${\bf
s}'_i, i \in I$ and ${\bf s}_\infty$ via
\[
({\bf s}'_i)^{d_i} = {\bf s}'_i {\bf s}_\infty = {\bf s}_\infty {\bf
s}'_i = {\bf s}_\infty^2 = {\bf s}_\infty,\ \forall i \in I,
\]
and the analogues for the ${\bf s}'_i$ of the braid word relations:
\[
{\bf v}'_j - {\bf w}'_j = 0, \ \forall j \in J_0, \qquad
{\bf v}'_l = {\bf w}'_l = {\bf s}_\infty, \forall l \in J \setminus J_0.
\]
Then $\widetilde{A}$ is indeed a bialgebra, under:
\[
\Delta({\bf s}'_i) := {\bf s}'_i \otimes {\bf s}'_i, \quad
\Delta({\bf s}_\infty) := {\bf s}_\infty \otimes {\bf s}_\infty, \quad
\varepsilon({\bf s}_i) = \varepsilon({\bf s}'_\infty) = 1, \qquad
\forall i \in I,
\]
and hence $\Rep \widetilde{A}$ is a monoidal category. We refer the
reader to \cite[Section 3]{Kh2} for more details in the special case $J_0
= J$.
Notice also that the algebras $\widetilde{A} =
\widetilde{n\mathcal{H}}(W, {\bf d}, J_0)$ form a lattice of bialgebras
that is isomorphic to the product of $2^J$ with $I$ copies of
$\Z^{\geqslant 2}$.\medskip

A second theme, extensively explored for decades in the Coxeter/Hecke/Lie
theory literature, involves classifying finite-dimensional objects in
various settings. Such classifications have been of enormous interest in
recent and earlier times -- including
complex simple Lie algebras;
real and complex reflection groups \cite{Cox,ST} and their nil-Coxeter
and more general Hecke algebras;
finite type quivers,
the McKay--Slodowy correspondence, and
Kleinian singularities.
This evergreen theme has seen recent additions, including for
finite-dimensional Nichols algebras \cite{GHV,HV1,HV2}; and
finite-dimensional pointed Hopf algebras \cite{AnSc}, which are
intimately connected to small quantum groups. There are also other
combinatorial-type results, including by Stembridge and Hart, that are
discussed presently; and the second author's recent work classifying the
finite-dimensional generalized nil-Coxeter algebras \cite{Kh2}. The
present work is a sequel to this last reference \cite{Kh2}, connecting it
with the works by Stembridge and Hart,
and also going beyond \cite{Kh2} via refined combinatorial phenomena for
$k=4,5$ even in the ``usual'' (i.e., ${\bf d} = (2,\dots,2)$) Coxeter
types $B_n, F_4, H_3, H_4$.\medskip

We next mention a connection with the question of \textit{flatness of
deformations}.
 In this paper, our goal is to explore the flatness question for e.g.\
the Temperley--Lieb algebras and their nil-Hecke versions. For instance,
in their theses Fan and Graham studied Temperley--Lieb algebras
$TL_W((2,\dots,2))$ for Coxeter groups $W$ (see \cite{Fan1,Fan2,Gr}), say
with $\bk = \Z[u,u^{-1}]$. These are quotients of the Iwahori--Hecke
algebras $\mathcal{H}_u(W)$; the authors showed in \textit{loc.~cit.}
that $TL_W((2,\dots,2))$ has a $\bk$-basis in bijection with the fully
commutative words in $W$. Now Theorem \ref{T3} shows that
$TL_W((2,\dots,2))$ is indeed a flat deformation of $NTL_W((2,\dots,2))$,
meaning that the following diagram is (i) valid, and (ii) a commuting
square:
\[ \begin{CD}
\mathcal{H}_u(W) @>{\rm gr} >> NC_W((2,\dots,2))\\
@VVV @VVV\\
TL_W((2,\dots,2)) @>{\rm gr} >> NTL_W((2,\dots,2))
\end{CD} \]

\noindent More strongly, we will show in Theorem \ref{Tdim} that the
above flatness phenomenon holds for nil-Hecke algebras
$n\mathcal{H}(W,(2,\dots,2),J_{<k})$ for \textit{all} $k$. We then
extend this to arbitrary ${\bf d}$ in Theorem~\ref{Tdimplus}.\medskip

Our next connection is to the combinatorics of Coxeter groups $W$.
Stembridge \cite{St1} studied the fully commutative elements $W_{\rm
fc}$(see Theorem~\ref{T3} for the definition) in connection with the
Bruhat and weak orderings. In \cite{KMMZ} the authors study the $2$-sided
Kazhdan--Lusztig cell formed by the words $W_1 \subset W$ that have a
unique reduced expression. These words have been counted recently in
\cite{Ha}. In fact all words in the group $W$ have been studied for the
above reasons.

We now introduce a more general notion of word-sets $W(J_0) \subset W$
for any Coxeter group, of which the above sets $W_{\rm fc}, W_1, W$ are
special cases for $J_0 = J_{<3}, J_{<2}, J_{<\infty}$ respectively.

\begin{definition}\label{Ddim}
Given a Coxeter matrix $M$ with corresponding group $W$, define for $J_0
\subset \binom{I}{2}$ the subset of elements $W(J_0) \subset W$, to
consist of all $w \in W$ such that no reduced expression for $w$ has a
substring of positive length that occurs in the set $\{ {\bf v}_j, {\bf
w}_j : j \in J \setminus J_0 \}$.
\end{definition}

With this notation in hand, we can state the promised result on flatness
of deformations of Temperley--Lieb algebras -- and more generally, relate
every subset $W(J_0)$ to the corresponding nil-Hecke algebra
$n\mathcal{H}(W, (2,\dots,2), J_0)$.

\begin{theorem}\label{Tdim}
Suppose $M$ is a Coxeter matrix, and $J_0 \subset \binom{I}{2}$ as above.
Then the nil-Hecke algebra $n\mathcal{H}(W, (2,\dots,2), J_0)$ is a free
$\bk$-module with basis $\{ \overline{{\bf s}_w} : w \in W(J_0) \}$,
where $\{ {\bf s}_w : w \in W \}$ is the ``canonical'' basis of the
nil-Coxeter algebra $NC_W((2,\dots,2))$.
\end{theorem}

\noindent In particular, by \cite{Fan1,Fan2,Gr}, Temperley--Lieb algebras
are flat deformations of their nil-Temperley--Lieb analogues.

Theorem \ref{Tdim} thus provides a broader setting into which the
families of words $W(J_0)$ fit: bases of nil-Hecke algebras corresponding
to Coxeter matrices. Even more generally, we will consider (bases of)
nil-Hecke algebras for generalized Coxeter matrices, i.e.~where ${\bf
s}_i^{d_i} = 0$ for $d_i \geqslant 2$. Beyond proving Theorem \ref{Tdim},
in the remaining sections we will characterize when such bases are finite
in size.

\begin{proof}[Proof of Theorem \ref{Tdim}]
Begin by observing that the subset $M$ in $NC_W(2,\dots,2))$ consisting
of the words
\[
M := \{ {\bf s}_w : w \in W \} \sqcup \{ 0 \}
\]
forms a monoid under multiplication; and moreover, $M \setminus \{ 0 \}$
forms a $\bk$-basis of $NC_W((2,\dots,2))$ (see e.g.\ \cite{Hum}). Now
let $I(J_0)$ be the $\bk$-span in $NC_W((2,\dots,2))$ of the words
\[
\{ \alpha {\bf v}_j \beta,\ \alpha {\bf w}_j \beta\ :\ j \in J \setminus
J_0, \ \alpha, \beta \in M \}.
\]
Then $I(J_0)$ clearly contains the elements ${\bf s}_w,\ w \not\in
W(J_0)$; moreover, it is a two-sided ideal of $NC_W((2,\dots,2))$. Now
suppose a linear combination $\sum_{w \in W} c_w {\bf s}_w \in I(J_0)$.
By definition/$\bk$-freeness, if $c_w \neq 0$ then $w \not\in W(J_0)$,
and $I(J_0)$ intersects the $\bk$-span of $\{ {\bf s}_w : w \in W(J_0)
\}$ trivially. This concludes the proof.
\end{proof}


The above theorem provides a natural basis for nil-Hecke algebras with
$\mathbf{d}=(2,2,\ldots ,2)$. We now strengthen this result to one which
applies to all of our nil-Hecke algebras, and which is key to proving the
finite-dimensionality in all of our main results.

\begin{theorem}\label{Tdimplus}
Let $W$ be a Coxeter group generated by $\{ s_i : i \in I \}$, with
corresponding nil-Hecke algebra $\nilhecke{<k}$, where ${\bf d} \in
(\Z^{\geqslant 2})^I$ and $1 \leqslant k \leqslant \infty$. Consider
equivalence classes of strings in the $s_i$, where two strings are
equivalent if one can be reached from another by applying the braid
relations finitely many times. Let $S$ denote the set of classes such
that no string in the class contains any of the substrings $s_i^{d_i}$ or
the braid words of length $k$ or more. Then the monomials corresponding
to some choice of class representatives from $S$ form a free $\bk$-module
basis of $\nilhecke{<k}$.
\end{theorem}

\begin{proof}
Since any monomial not coming from $S$ can be reduced to zero, these
monomials clearly span $\nilhecke{<k}$. It remains to prove that $\dim
\nilhecke{<k} \geqslant |S|$, and we would be done.

Consider a free $\bk$-module $\mathcal M$ with basis elements indexed by
$S$. We now define a natural $\nilhecke{<k}$-module structure on
$\mathcal M$ as follows: the element $\mathbf{s}_j$ acts on the class
$[s_{i_1}\cdots s_{i_n}]$ by sending it to the class $[s_js_{i_1}\cdots
s_{i_n}]$ if it is in $S$, and to zero otherwise. This extends to an
action of the free associative algebra $\bk\langle \mathbf s_i\mid i\in
I\rangle$.

We now verify this action indeed satisfies the defining relations in
$\nilhecke{<k}$. Indeed, if applying the $\mathbf{s}_i$-action $d_i$
times on a basis element yields a non-zero basis element, then this class
contains a string with $d_i$ consecutive leading $s_i$'s and thus cannot
be in $S$, a contradiction. Similarly, multiplying by braid words of
length $k$ or more gives the zero vector. Finally, the braid relations of
lengths less than $k$ hold: indeed, multiplying a basis element by either
side of the braid equations either yields elements in the same
equivalence class by definition, or yields elements not in $S$, giving
the zero vector on either side.

Since $\mathcal M$ is generated as a $\nilhecke{<k}$-module by the basis
vector corresponding to the empty string, this gives a surjection
$\nilhecke{<k}\twoheadrightarrow\mathcal M$. This, combined with the
first paragraph, completes our proof.
\end{proof}

Following Theorem~\ref{Tdimplus} on a word basis of $\nilhecke{<k}$, we
describe here a final connection of this work, to the study of PBW bases
and deformations -- more broadly than the commuting square in this
section. The study of PBW (Poincar\'e--Birkhoff--Witt) type deformations
has seen intense activity for several decades now, including foundational
works on Drinfeld Hecke algebras \cite{Dr}, graded affine Hecke algebras
\cite{Lu}, symplectic reflection algebras \cite{EG} (including rational
Cherednik algebras), as well as on infinitesimal Hecke algebras, quantum
analogues, and their generalizations. The recent program of Shepler and
Witherspoon (and their coauthors) has led to a profusion of activity; we
mention e.g.\ \cite{SW2,SW4,WW} here. In all of these works, a bialgebra
$A$ (which is most often a Hopf algebra) acts on a symmetric algebra
$S_V$ of some vector space $V$, and one is interested in understanding
which deformations of the smash-product algebra $A \ltimes S_V$ are
``PBW'', or flat.

When $A$ is a nil-Hecke algebra, there are several significant features
to note. First, these are not bialgebras by Proposition~\ref{Pnotbialg}
(but possess a coproduct $\Delta$); and yet, a variant of the ``PBW
theorem'' nevertheless holds for the algebras $A \ltimes S_V$, which
subsumes the PBW theorems in various previous works. (See the previous
paper \cite{Kh} by the second author for the details.) Thus the nil-Hecke
algebras $\nilhecke{0}$ widen beyond \cite{Kh2} the class of examples of
such ``algebras with coproduct'' $(A,\Delta)$, which do not possess an
antipode or even a counit, yet fit into the framework of the
aforementioned works.

Second, there also are technical consequences of finite-dimensionality.
We mention two of these; in both, we will assume $\bk$ to be a field. It
was shown in \cite{SW4}, \cite{Kh} that if $(A,\Delta)$ is
finite-dimensional, one can characterize those graded
$\bk[t]$-deformations of $A \ltimes S_V$, for which the fiber at $t=1$
satisfies the PBW property. This provides a PBW-theoretic motivation to
classify the finite-dimensional nil-Hecke algebras.

Even more is true. It was shown in \cite{Kh} that if $(A,\Delta)$ is any
algebra with coproduct (e.g. $A = \nilhecke{<k}$), and if it is local
with augmentation ideal $\m$ such that $\Delta(\m) \subset \m \otimes
\m$, then one obtains much additional information regarding the
deformations of $A \ltimes S_V$ -- their abelianization, center, and
(simple) modules. See e.g.\ \cite[Section 6.1]{Kh}. Thus, we have a
second motivation from the perspective of PBW deformations -- to
understand which nil-Hecke algebras $\nilhecke{<k}$ are local. This is
precisely the point of Theorem~\ref{TNil}; and it shows that this
question turns out to be \textit{equivalent} to the main results of the
paper: $\nilhecke{<k}$ must be finite-dimensional.

\section{Proof of Theorem \ref{T12}:
finite rank nil-Hecke algebras with $k=1,2$}

The remainder of the paper is devoted to proving the above classification
theorems on finite-dimensional nil-Hecke algebras. To do so, we will
employ the diagrammatic calculus utilized in \cite{Kh2}. We begin by
showing Theorem \ref{T12} in this section; the proof is in steps.\medskip

\noindent \textbf{Step 1:}
By Theorem \ref{Tdim}, $\nilhecke{<2}$ has finite $\bk$-rank if and only
if $W(J_{<2})$ is also finite. By using the results of \cite{Ha}, this
proves the result if ${\bf d} = (2,\dots,2)$, i.e., when dealing with
quotients of usual nil-Coxeter algebras.\medskip

\noindent \textbf{Step 2:}
In the remainder of the proof, at least one $d_i$ is $3$ or more. Since
$\nilhecke{0} \twoheadrightarrow n\mathcal{H}(W, (2,\dots,2), J_0)$, by
the previous step the Coxeter (multi)graph of $W$ must be a tree with no
$m_{ij} = \infty$ and at most one $m_{ij} \geqslant 4$.

Suppose first that we have a simply laced tree, and $d_\alpha, d_\gamma
\geqslant 3$ for $\alpha \neq \gamma \in I$. As the graph of $W$ is
connected, suppose
\[
\alpha \quad \longleftrightarrow \quad \beta_1 \quad \longleftrightarrow
\quad \cdots \quad \longleftrightarrow \quad \beta_{m-1} \quad
\longleftrightarrow \quad \gamma
\]

\noindent is a path in $I$ (in the figures below, we write $m' = m-1$).
Now define a free $\bk$-module $\scrm$ with basis vectors
\[
\{ A_r, B_{1r}, \dots, B_{mr}, C_r, B'_{1r}, \dots, B'_{mr} : r \geqslant
1 \}.
\]

\noindent Give $\scrm$ an $\nilhecke{<2}$-module structure by letting
every ${\bf s}_i$ kill all basis vectors, except for the actions
described by the first diagram in Figure~\ref{Fig2}.
Explicitly, the basis vectors on which the action is nonzero are:

\begin{figure}[ht]
\begin{tikzpicture}[line cap=round,line join=round,>=triangle 45,x=1.0cm,y=1.0cm]
\draw(5.8,7.7) circle (0.25cm);
\draw(7.5,8.8) circle (0.4cm);
\draw(9.5,8.8) circle (0.4cm);
\draw(11.5,8.8) circle (0.4cm);
\draw(13.5,8.8) circle (0.4cm);
\draw(13.5,6.6) circle (0.4cm);
\draw(11.5,6.6) circle (0.4cm);
\draw(9.5,6.6) circle (0.4cm);
\draw(7.5,6.6) circle (0.4cm);
\draw(14.9,7.7) circle (0.25cm);
\draw (5.5,8) node[anchor=north west] {$A$};
\draw (7.1,9.1) node[anchor=north west] {$B_1$};
\draw (9.1,9.1) node[anchor=north west] {$B_2$};
\draw (11,9.1) node[anchor=north west] {$B_{m'}$};
\draw (13.05,9.1) node[anchor=north west] {$B_m$};
\draw (13,6.9) node[anchor=north west] {$B'_m$};
\draw (11,6.9) node[anchor=north west] {$B'_{m'}$};
\draw (9.1,6.9) node[anchor=north west] {$B'_2$};
\draw (7.1,6.9) node[anchor=north west] {$B'_1$};
\draw (14.6,8) node[anchor=north west] {$C$};
\draw [->] (7,6.9) -- (6.1,7.5);
\draw [->] (6.1,7.9) -- (7,8.7);
\draw [->] (8,8.8) -- (9,8.8);
\draw [->] (12,8.8) -- (13,8.8);
\draw [->] (13,6.6) -- (12,6.6);
\draw [->] (9,6.6) -- (8,6.6);
\draw [->] (13.9,8.5) -- (14.6,7.9);
\draw [->] (14.6,7.5) -- (13.9,6.8);
\draw (6.2,7.2) node[anchor=north west] {$ \alpha $};
\draw (6.2,8.7) node[anchor=north west] {$ \alpha $};
\draw (8.1,9.4) node[anchor=north west] {$ \beta_1 $};
\draw (10.1,9) node[anchor=north west] {$ \cdots $};
\draw (12,9.4) node[anchor=north west] {$ \beta_{m'} $};
\draw (14.2,8.8) node[anchor=north west] {$ \gamma $};
\draw (14.2,7.2) node[anchor=north west] {$ \gamma $};
\draw (12.3,6.6) node[anchor=north west] {$ \beta_{m'} $};
\draw (10.1,6.8) node[anchor=north west] {$ \cdots $};
\draw (8.3,6.6) node[anchor=north west] {$ \beta_1 $};
\draw (6,8) node[anchor=north west] {+};
%
%
\draw(5.8,2.1) circle (0.25cm);
\draw(7.5,3.2) circle (0.4cm);
\draw(9.5,3.2) circle (0.4cm);
\draw(11.5,3.2) circle (0.4cm);
\draw(13.5,3.2) circle (0.4cm);
\draw(13.5,1) circle (0.4cm);
\draw(11.5,1) circle (0.4cm);
\draw(9.5,1) circle (0.4cm);
\draw(7.5,1) circle (0.4cm);
\draw (5.5,2.4) node[anchor=north west] {$A$};
\draw (7.1,3.5) node[anchor=north west] {$B_1$};
\draw (9.1,3.5) node[anchor=north west] {$B_2$};
\draw (11,3.5) node[anchor=north west] {$B_{m'}$};
\draw (13.05,3.5) node[anchor=north west] {$B_m$};
\draw (13,1.3) node[anchor=north west] {$B'_m$};
\draw (11,1.3) node[anchor=north west] {$B'_{m'}$};
\draw (9.1,1.3) node[anchor=north west] {$B'_2$};
\draw (7.1,1.3) node[anchor=north west] {$B'_1$};
\draw [->] (7,1.3) -- (6.1,1.9);
\draw [->] (6.1,2.3) -- (7,3.1);
\draw [->] (8,3.2) -- (9,3.2);
\draw [->] (12,3.2) -- (13,3.2);
\draw [->] (13.5,2.7) -- (13.5,1.5);
\draw [->] (13,1) -- (12,1);
\draw [->] (9,1) -- (8,1);
\draw (6.2,1.6) node[anchor=north west] {$ \alpha $};
\draw (6.2,3.1) node[anchor=north west] {$ \alpha $};
\draw (8.1,3.8) node[anchor=north west] {$ \beta_1 $};
\draw (10.1,3.4) node[anchor=north west] {$ \cdots $};
\draw (12,3.8) node[anchor=north west] {$ \beta_{m'} $};
\draw (13.6,2.4) node[anchor=north west] {$ \gamma $};
\draw (12.3,1) node[anchor=north west] {$ \beta_{m'} $};
\draw (10.1,1.2) node[anchor=north west] {$ \cdots $};
\draw (8.3,1) node[anchor=north west] {$ \beta_1 $};
\draw (6,2.4) node[anchor=north west] {+};
\end{tikzpicture}
\caption{The modules $\scrm$ for the nil-Hecke algebras in Steps 2 and 3}
\label{Fig2}
\end{figure}
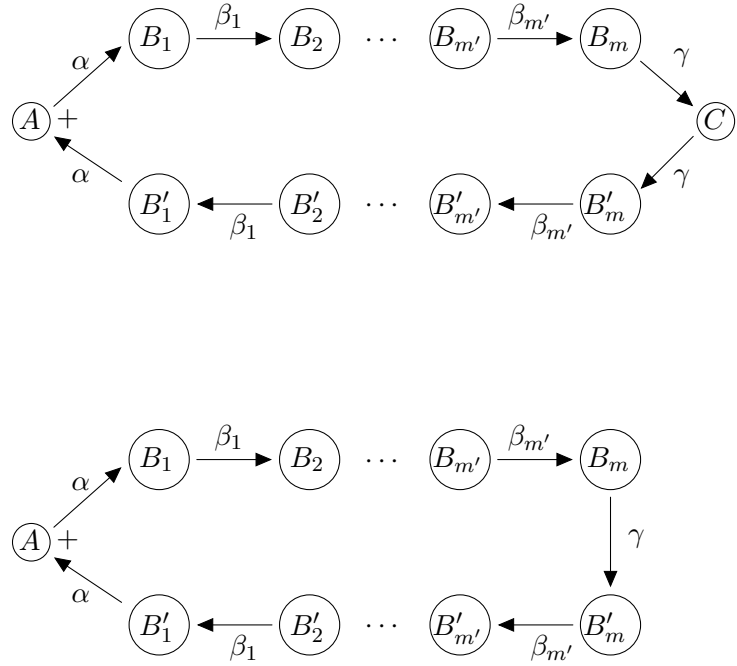

\begin{align*}
{\bf s}_\alpha(A_r) := &\ B_{1r}, \quad 
{\bf s}_{\beta_j}(B_{jr}) := B_{j+1,r}\ (1 \leqslant j \leqslant m-1),
\quad {\bf s}_\gamma(B_{mr}) := C_r,\\
{\bf s}_\gamma(C_r) := &\ B'_{mr}, \quad {\bf s}_{\beta_j}(B_{jr}) :=
B_{j-1,r}\ (2 \leqslant j \leqslant m),
\quad {\bf s}_\alpha(B'_{1r}) := A_{r+1}.
\end{align*}

\noindent Thus the ``+'' at the head of an arrow refers to the index $r$
increasing by $1$. One verifies that the defining relations for
$\nilhecke{<2}$ hold on each basis vector, hence in ${\rm
End}_\bk(\scrm)$. Thus $\scrm$ is a cyclic $\nilhecke{<2}$-module
generated by $A_{11}$. Now $\scrm$, hence $\nilhecke{<2}$, is not a
finitely generated $\bk$-module.\medskip

\noindent \textbf{Step 3:}
Suppose next that some $m_{ij} > 3$ and some $d_{i_0} > 2$. Without loss
of generality there exist nodes nodes $\beta_1, \dots, \beta_{m-1}$ for
some $m \geqslant 1$ (in the figure we write $m' := m-1$), such that
\[
\alpha = i_0 \quad \longleftrightarrow \quad \beta_1 \quad
\longleftrightarrow \quad \cdots \quad \longleftrightarrow \quad
\beta_{m-1} = i \quad \longleftrightarrow \quad \gamma = j
\]

\noindent Now construct a free $\bk$-module $\scrm$ as in Step 2, this
time giving it an $\nilhecke{<2}$-module structure using the second
diagram in Figure~\ref{Fig2}. One verifies that the action on $\scrm$
satisfies the defining relations. Hence the nil-Hecke algebra is once
again of infinite $\bk$-rank.\medskip

\noindent \textbf{Step 4:}
It remains to show the assertion when the Coxeter (multi)graph of $W$ is
a simply laced tree with node set $I$, and $d_{i_0} \geqslant 3$ for a
unique $i_0 \in I$. Denote the corresponding nil-Hecke algebra by
$n\mathcal{H}(W,i_0,d_{i_0}, J_{<2})$. For nodes $s,t \in I$, let $[s,t]$
denote the path from source $s \in I$ to target $t \in I$. If $s = i_1,
\dots, i_n = t$ enumerates sequentially the nodes in $[s,t]$ for some $n
\geqslant 1$, define:
\begin{equation}
{\bf s}_{[s,t]} := {\bf s}_{i_1} \cdots {\bf s}_{i_n}.
\end{equation}

We begin by enumerating a spanning set for $n\mathcal{H}(W,i_0,d_{i_0},
J_{<2})$ of the correct size. Notice that the product ${\bf s}_i {\bf
s}_j$ (with $i,j \in I$) vanishes unless $i=j=i_0$ or $i \neq j$ are
adjacent in $I$. It follows that the set
\[
\{ 1 \} \sqcup \{ {\bf s}_{[s,t]} : s,t \in I \} \sqcup
\{ {\bf s}_{[s,i_0]} {\bf s}_{i_0}^{k-2} {\bf s}_{[i_0,t]} : s,t \in I, \
2 \leqslant k \leqslant d_{i_0} - 1 \}
\]

\noindent spans $n\mathcal{H}(W, i_0, d_{i_0}, J_{<2})$. Now by the proof
of Theorem~\ref{Tdimplus}, this is indeed a basis, and we are done.\qed


\section{Proof of Theorem \ref{T3}:
finite rank nil-Hecke algebras with $k=3$}

\begin{proof}[Proof of Theorem \ref{T3}]
We begin by showing that the only ``adjacency graphs'', corresponding to
Coxeter groups $W$ that admit finite $\bk$-rank nil-Hecke algebras for
$NTL_W({\bf d}) = \nilhecke{<3}$ (these algebras are defined following
Definition~\ref{D11}), are Dynkin diagrams of finite type or type $E_n, n
\geqslant 9$. Here, first, by the \textit{adjacency multigraph} of a
Coxeter group associated to a Coxeter matrix $(m_{ij})_{i,j=1}^n$ with $n
= |I|$, we mean the multigraph on $n$ nodes with $m_{ij} - 2$ edges
between nodes $i \neq j \in \{ 1, \dots, n \}$. Next, the corresponding
\textit{adjacency graph} is the simple graph on $n$ nodes with ${\bf
1}(m_{ij} > 2)$ edges between nodes $i \neq j$.

If ${\bf d}=(2,\ldots, 2)$, then $W$ is a Coxeter group, and by Theorem
\ref{Tdim}, the problem reduces to finding all $W$ with finitely many
fully commutative elements. This was worked out in \cite{Fan0,Gr,St1}. In
what follows, we therefore assume there exists at least one $i$ such that
$d_i\geqslant 3$. Let this $i$ correspond to vertex $\alpha$.\medskip

\noindent \textbf{Step 1.}
Suppose the adjacency multigraph of $W$ contains a cycle on nodes
$\beta_1, \dots, \beta_m$ with $m \geqslant 3$, i.e., $m_{\beta_i,
\beta_{i+1}} > 2$ for $0<i<m$ and $m_{\beta_m, \beta_1} > 2$. Now
construct a free $\bk$-module $\scrm$ with basis given by the countable
set $\{ A_{1r}, \dots, A_{mr} : r \geqslant 1 \}$, and define the
following $NTL_W({\bf d})$-action on it: ${\bf s}_{\beta_i}$ kills
$A_{jr}$ except that ${\bf s}_{\beta_i}(A_{ir}) := A_{i+1,r}$ for $0 < i
< r$ and ${\bf s}_{\beta_m}(A_{mr}) := A_{1,r+1}$. It is easy to verify
that the defining relations of $NTL_W({\bf d})$ hold in ${\rm
End}_{\bk}(\scrm)$, as they hold on each $A_{ir}$. Therefore $\scrm$ is a
cyclic $NTL_W({\bf d})$-module generated by $A_{11}$, which is not
finitely generated as a $\bk$-module. Hence neither is $NTL_W({\bf d})$,
as claimed.

The analysis in this step can be conveniently expressed by a diagram of a
cycle. We will do so for other cases in the remainder of the
proof.\medskip

\noindent \textbf{Step 2.}
The remaining adjacency multigraphs are those whose underlying simple
graphs are trees. We next claim there is no vertex adjacent to $4$ nodes,
and at most one vertex adjacent to $3$ nodes.
Indeed, if $\beta_1$ is adjacent to nodes $\alpha_1, \alpha_2, \gamma_1,
\gamma_2$, then we appeal to Figure \ref{Fig6} with $m=2$, to generate a
$NTL_W({\bf d})$-module $\scrm$ with basis $\{ A_r, B_r, B_{1r}, B_{2r},
B'_{1r}, B'_{2r}, C_r, D_r : r \geqslant 1 \}$. The module relations are
read off of the diagram. Namely, ${\bf s}_i$ kills all basis vectors for
all $i \in I$, with the following exceptions:
\begin{align*}
& {\bf s}_{\beta_1}(B_{1r}) = B_{2r}, \quad
{\bf s}_{\gamma_1}(B_{2r}) = D_r, \quad
{\bf s}_{\gamma_2}(B_{2r}) = C_r, \quad
{\bf s}_{\gamma_1}(C_r) = B'_{2r}, \quad
{\bf s}_{\gamma_2}(D_r) = B'_{2r},\\
& {\bf s}_{\beta_1}(B'_{2r}) = B'_{1r}, \quad
{\bf s}_{\alpha_2}(B'_{1r}) = B_r, \quad
{\bf s}_{\alpha_1}(B_r) = B_{1r}, \quad
{\bf s}_{\alpha_2}(A_r) = B_{1r}, \quad
{\bf s}_{\alpha_1}(B'_{1r}) = A_{r+1}.
\end{align*}

\begin{figure}[ht]
\begin{tikzpicture}[line cap=round,line join=round,>=triangle 45,x=1.0cm,y=1.0cm]
\draw(5.8,7.7) circle (0.25cm);
\draw(9,7.7) circle (0.25cm);
\draw(7.5,8.8) circle (0.4cm);
\draw(9.5,8.8) circle (0.4cm);
\draw(11.5,8.8) circle (0.4cm);
\draw(13.5,8.8) circle (0.4cm);
\draw(13.5,6.6) circle (0.4cm);
\draw(11.5,6.6) circle (0.4cm);
\draw(9.5,6.6) circle (0.4cm);
\draw(7.5,6.6) circle (0.4cm);
\draw(12.1,7.7) circle (0.25cm);
\draw(14.9,7.7) circle (0.25cm);
\draw (5.5,8) node[anchor=north west] {$A$};
\draw (8.7,8) node[anchor=north west] {$B$};
\draw (7.1,9.1) node[anchor=north west] {$B_1$};
\draw (9.1,9.1) node[anchor=north west] {$B_2$};
\draw (11,9.1) node[anchor=north west] {$B_{m'}$};
\draw (13.05,9.1) node[anchor=north west] {$B_m$};
\draw (13,6.9) node[anchor=north west] {$B'_m$};
\draw (11,6.9) node[anchor=north west] {$B'_{m'}$};
\draw (9.1,6.9) node[anchor=north west] {$B'_2$};
\draw (7.1,6.9) node[anchor=north west] {$B'_1$};
\draw (11.8,8) node[anchor=north west] {$D$};
\draw (14.6,8) node[anchor=north west] {$C$};
\draw [->] (7.1,6.8) -- (6.1,7.5);
\draw [->] (6.1,7.9) -- (7,8.7);
\draw [->] (7.9,6.8) -- (8.8,7.5);
\draw [->] (8.8,8) -- (8,8.7);
\draw [->] (8,8.8) -- (9,8.8);
\draw [->] (12,8.8) -- (13,8.8);
\draw [->] (13,6.6) -- (12,6.6);
\draw [->] (9,6.6) -- (8,6.6);
\draw [->] (13.1,8.6) -- (12.3,7.9);
\draw [->] (12.3,7.5) -- (13.1,6.9);
\draw [->] (13.9,8.5) -- (14.6,7.9);
\draw [->] (14.6,7.5) -- (13.9,6.8);
\draw (6.2,7.2) node[anchor=north west] {$ \alpha_1 $};
\draw (6,8.7) node[anchor=north west] {$ \alpha_2 $};
\draw (7.8,7.6) node[anchor=north west] {$ \alpha_2 $};
\draw (7.8,8.5) node[anchor=north west] {$ \alpha_1 $};
\draw (8.1,9.4) node[anchor=north west] {$ \beta_1 $};
\draw (10.1,9) node[anchor=north west] {$ \cdots $};
\draw (12,9.4) node[anchor=north west] {$ \beta_{m'} $};
\draw (12.7,8.4) node[anchor=north west] {$ \gamma_1 $};
\draw (12.7,7.6) node[anchor=north west] {$ \gamma_2 $};
\draw (14.2,8.8) node[anchor=north west] {$ \gamma_2 $};
\draw (14.2,7.2) node[anchor=north west] {$ \gamma_1 $};
\draw (12.3,6.6) node[anchor=north west] {$ \beta_{m'} $};
\draw (10.1,6.8) node[anchor=north west] {$ \cdots $};
\draw (8.3,6.6) node[anchor=north west] {$ \beta_1 $};
\draw (6,8) node[anchor=north west] {+};
\draw (8.2,8) node[anchor=north west] {+};
\end{tikzpicture}
\caption{Diagrammatic calculus for the algebras $NTL(M)$; here $m' =
m-1$.}
\label{Fig6}
\end{figure}
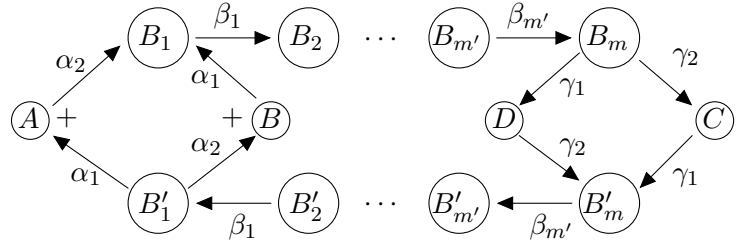

\noindent (The ``+'' at the head of an arrow again refers to the index
increasing by $1$.) It is easy to verify that the defining relations of
$NTL_W({\bf d})$ hold at each node in the diagram (i.e., at the
corresponding basis vector) and hence in ${\rm End}_\bk(\scrm)$. This
yields a cyclic $NTL_W({\bf d})$-module $\scrm$ of infinite $\bk$-rank,
as desired.

Next if $\deg \alpha = \deg \gamma = 3$, let
\[
\alpha = \beta_1 \quad \longleftrightarrow \quad \beta_2 \quad
\longleftrightarrow \quad \cdots \quad \longleftrightarrow \quad
\beta_{m-2} \quad \longleftrightarrow \quad \gamma = \beta_{m-1}
\]

\noindent be the path between them, for suitable $m$. Also suppose
$\alpha$ is adjacent to $\beta_2, \alpha_1, \alpha_2$, and $\gamma$ is
adjacent to $\beta_{m-2}, \gamma_1, \gamma_2$. We again appeal to Figure
\ref{Fig6} to generate a representation $\scrm$ which has infinite
$\bk$-rank.

\begin{remark}
A brief remark for the reader, about checking that the relations hold in
Figure \ref{Fig6} and the other figures in this paper:
(a)~On each commuting sub-diagram, one only needs to check the braid
relations in $J_{<k}$, and account for the `$+$' signs.
(b)~On each directed cycle in the diagrams, one needs to check that no
braid word in $J \setminus J_{<k}$ occurs, nor any word ${\bf
s}_i^{d_i}$.
\end{remark}

\noindent \textbf{Step 3.}
Thus the Dynkin multigraph is in fact acyclic, with at most one vertex of
degree $3$ and no vertex of degree $\geqslant 4$. First, we use the first
diagram in Figure~\ref{Fig2} to show there exists a unique node $i$ such
that $d_i \geqslant 3$. Second, use Figure~\ref{FigS3} with $m=1$ to show
this node has degree $1$ (i.e., is pendant).

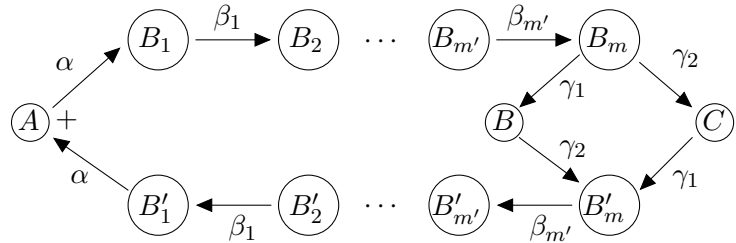
\begin{figure}[ht]
\begin{tikzpicture}[line cap=round,line join=round,>=triangle 45,x=1.0cm,y=1.0cm]
\draw(5.8,7.7) circle (0.25cm);
\draw(7.5,8.8) circle (0.4cm);
\draw(9.5,8.8) circle (0.4cm);
\draw(11.5,8.8) circle (0.4cm);
\draw(13.5,8.8) circle (0.4cm);
\draw(13.5,6.6) circle (0.4cm);
\draw(11.5,6.6) circle (0.4cm);
\draw(9.5,6.6) circle (0.4cm);
\draw(7.5,6.6) circle (0.4cm);
\draw(12.1,7.7) circle (0.25cm);
\draw(14.9,7.7) circle (0.25cm);
\draw (5.5,8) node[anchor=north west] {$A$};
\draw (7.1,9.1) node[anchor=north west] {$B_1$};
\draw (9.1,9.1) node[anchor=north west] {$B_2$};
\draw (11,9.1) node[anchor=north west] {$B_{m'}$};
\draw (13.05,9.1) node[anchor=north west] {$B_m$};
\draw (13,6.9) node[anchor=north west] {$B'_m$};
\draw (11,6.9) node[anchor=north west] {$B'_{m'}$};
\draw (9.1,6.9) node[anchor=north west] {$B'_2$};
\draw (7.1,6.9) node[anchor=north west] {$B'_1$};
\draw (11.8,8) node[anchor=north west] {$B$};
\draw (14.6,8) node[anchor=north west] {$C$};
\draw [->] (7.1,6.8) -- (6.1,7.5);
\draw [->] (6.1,7.9) -- (7,8.7);
\draw [->] (8,8.8) -- (9,8.8);
\draw [->] (12,8.8) -- (13,8.8);
\draw [->] (13,6.6) -- (12,6.6);
\draw [->] (9,6.6) -- (8,6.6);
\draw [->] (13.1,8.6) -- (12.3,7.9);
\draw [->] (12.3,7.5) -- (13.1,6.9);
\draw [->] (13.9,8.5) -- (14.6,7.9);
\draw [->] (14.6,7.5) -- (13.9,6.8);
\draw (6.2,7.2) node[anchor=north west] {$ \alpha $};
\draw (6,8.7) node[anchor=north west] {$ \alpha $};
%
%
\draw (8.1,9.4) node[anchor=north west] {$ \beta_1 $};
\draw (10.1,9) node[anchor=north west] {$ \cdots $};
\draw (12,9.4) node[anchor=north west] {$ \beta_{m'} $};
\draw (12.7,8.4) node[anchor=north west] {$ \gamma_1 $};
\draw (12.7,7.6) node[anchor=north west] {$ \gamma_2 $};
\draw (14.2,8.8) node[anchor=north west] {$ \gamma_2 $};
\draw (14.2,7.2) node[anchor=north west] {$ \gamma_1 $};
\draw (12.3,6.6) node[anchor=north west] {$ \beta_{m'} $};
\draw (10.1,6.8) node[anchor=north west] {$ \cdots $};
\draw (8.3,6.6) node[anchor=north west] {$ \beta_1 $};
\draw (6,8) node[anchor=north west] {+};
\end{tikzpicture}
\caption{Diagrammatic calculus for Step 3; here $m' = m-1$.}
\label{FigS3}
\end{figure}

Third, use the second diagram in Figure~\ref{Fig2} to show that the
Dynkin multigraph of $W$ is in fact simply laced. Hence from above, it is
a tree graph.

Next, suppose there is a vertex $\beta$ of degree $3$. Let
\[
\alpha = \beta_1 \quad \longleftrightarrow \quad \beta_2 \quad
\longleftrightarrow \quad \cdots \quad \longleftrightarrow \quad
\beta_{m-2} \quad \longleftrightarrow \quad \beta = \beta_{m-1}
\]
be a path from $\alpha$ to $\beta$; further, let $\gamma_1$ and
$\gamma_2$ be two more vertices adjacent to $\beta$. Then the module
pictured in Figure \ref{FigS3} again has infinite $\bk$-rank, a
contradiction. This shows that the adjacency multigraph is necessarily a
path graph. As we have seen, ${\bf d} = (d,2,\dots,2)$ for some $d
\geqslant 3$, and so we obtain a quotient of the algebra $NC_A(n,d)$ as
desired.\medskip

\noindent\textbf{Step 4.}
To finish the proof, we need to determine the $\bk$-rank of $NTL_W({\bf
d})$, where $W$ is of type $A_n$ and ${\bf d} = (d,2,\dots,2)$. By
Theorem~\ref{Tdimplus}, it suffices to enumerate all the monomials in
$x_1,\ldots, x_n$ that are not equivalent (via commutation relations
$x_ix_j=x_jx_i$ for $|i-j|>1$) to any monomial containing the strings
$x_1^d$, $x_i^2$ for $i>1$, or containing any of the length $3$ braid
words. We extend the argument in \cite{GK} to accomplish this.

Consider a nonzero monomial in $NTL_W({\bf d})$, and in its equivalence
class under the commutation relations ($x_i x_j = x_j x_i$ for
$|i-j|>1$), look at the lexicographically smallest monomial ${\bf w}$ as
in \cite{GK}. We claim that all the $x_1$'s occur in a contiguous string
in ${\bf w}$. Indeed, suppose not. Then by the arguments in \cite{GK},
the monomial ${\bf w}$ can be assumed to be composed of decreasing runs,
where in each run the indices go down by one, except possibly where $x_1$
is followed by $x_1$. Since $x_1$'s can only occur at the ends of such
runs, if there are two $x_1$'s that are not consecutive, a closest such
pair must be part of a substring that looks like $x_1x_jx_{j-1}\cdots
x_2x_1$. But this can be reduced to $x_jx_{j-1}\cdots x_1x_2x_1$ and thus
zero, proving our claim.

One can repeat the arguments in \cite[Section 3]{GK} and verify
that Lemmas 1, 3, and 4 there still hold, with the following exceptions:
\begin{enumerate}
\item There might be $x_1$ followed by $x_1$ in the monomial ${\bf w}$,
contrary to Lemma 1; or
\item The indices of the peaks might not be strictly increasing,
precisely when the first run starts with $x_i$ for $i\ge 2$, ends with
$x_1^j$ for $j\ge 2$, and is followed by $x_2$, contrary to Lemma 2.
\end{enumerate}

Let us set aside the monomials that come under exception (2) for the
moment. Then the monomials we need to count are precisely those counted
by \cite{GK}, with the added caveat that there might be a contiguous
block of up to $d-1$ $x_1$'s. Now if the monomial contains no $x_1$, then
this is simply a monomial in $NTL_W((2,\ldots, 2))$ on the generators
$x_2,\ldots, x_n$; by the results of \cite{GK}, there are $C_n$ such
monomials. If the monomial does contain a nonempty block of $x_1$'s, then
deleting all but one $x_1$ gives a monomial containing $x_1$ in
$NTL_W((2,\ldots, 2))$; conversely, given a monomial in $NTL_W((2,\ldots,
2))$ containing $x_1$, one can form $d-1$ different valid monomials from
there by adding $x_1$'s. The number of $NTL_W((2,\ldots, 2))$-monomials
containing an $x_1$ is simply $C_{n+1}-C_n$ (since there are $C_n$
monomials not having $x_1$), so the total number is
\[
C_n+(d-1)(C_{n+1}-C_n)=(d-1)C_{n+1}-(d-2)C_n.
\]

Finally, we count the monomials that exhibit exception (2). Such a
monomial must start with $x_ix_{i-1}\cdots x_2x_1^jx_2$, and then either
the next term is $x_3$, or the remaining string is a non-zero monomial on
the generators $x_{i+1},\dots, x_n$. Indeed, the peak of the next run
cannot be $x_k$ for $4\le k\le i$, else we can move it to the left by the
commutation relations and the entire monomial reduces to zero. The next
run cannot even contain such a $x_k$, else it would have to contain $x_i$
and thus the initial segment of this run till $x_i$ can be commuted
leftward to reduce the monomial to zero, proving our claim. Continuing
this argument, we see that the $x_1^j$ in our monomial is followed by a
string of the form $x_2x_3\cdots x_\ell$, with $2\le \ell \le i$, and a
nonzero (possibly empty) monomial in the generators $x_{i+1},\dots, x_n$.
Thus the number of such monomials is $(i-1)C_{n-i+1}$. Summing over all
possible $i$ and $j$, the total number becomes
\[
(d-2)\sum_{i=2}^{n}(i-1)C_{n-i+1}=(d-2)\sum_{j=1}^{n-1}jC_{n-j}.
\]
As in \cite{GK}, we can verify these indeed correspond to non-zero
monomials, so adding this to the previous count completes the proof.
\end{proof}

\section{Proof of Theorem \ref{T4} and its corollaries: $k \geq 4$}

We now consider the remaining cases $4 \leqslant k \leqslant \infty$,
starting with $k=4$.

\begin{proof}[Proof of Theorem \ref{T4}]
Since $\nilhecke{<\infty}$ surjects onto $\nilhecke{<4}$, it follows that
$\nilhecke{<4}$ has finite $\bk$-rank if $\nilhecke{<\infty}$ does; these
cases are listed in Theorem \ref{Tinfty}. On the other hand,
$\nilhecke{<4}\twoheadrightarrow\nilhecke{<3}$, so $\nilhecke{<4}$ has
infinite $\bk$-rank whenever $\nilhecke{<3}$ does. The cases where this
is not so are listed in Theorem \ref{T3}. Comparing these two lists, we
see that only the cases where $W$ is of type $E_n (n\geqslant 9)$, $F_n
(n\geqslant 5)$ or $H_n(n\geqslant 5)$ with ${\bf d}=(2,\ldots, 2)$
remain to be checked for finiteness.

If $W$ has type $E_n$, there are no braid relations of length $4$ or
more, and so $\nilhecke{<4} \cong \nilhecke{<\infty}$, which is known by
\cite{Kh2} to have infinite $\bk$-rank for $n\geqslant 9$. Next, assume
$W$ is of type $F_n$ with $n \geqslant 5$. We will prove that
$\nilhecke{<4}$ has infinite $\bk$-rank for $n=5$, from which the $n>5$
cases would follow trivially.

Let $W$ be of type $F_5$, and label the generators as in Figure \ref{F5}.
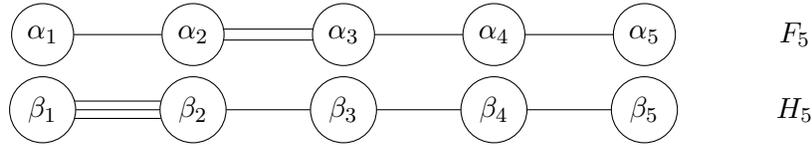
\begin{figure}[ht]
\begin{tikzpicture}[line cap=round,line join=round,>=triangle 45,x=2.0cm,y=1.0cm]
\tikzstyle{every circle node}=[draw]
\draw(1,0) node[circle](1){$\alpha_1$};
\draw(2,0) node[circle](2){$\alpha_2$};
\draw(3,0) node[circle](3){$\alpha_3$};
\draw(4,0) node[circle](4){$\alpha_4$};
\draw(5,0) node[circle](5){$\alpha_5$};
\draw(6,0) node[] {$F_5$};
\draw
(2.-10) edge[auto=right,-]  (3.190)  
(3.170) edge[auto=right,-]  (2.10);  
\draw[-] (1)--(2) (3)--(4) (4)--(5);

\draw(1,-1) node[circle](H1){$\beta_1$};
\draw(2,-1) node[circle](H2){$\beta_2$};
\draw(3,-1) node[circle](H3){$\beta_3$};
\draw(4,-1) node[circle](H4){$\beta_4$};
\draw(5,-1) node[circle](H5){$\beta_5$};
\draw(6,-1) node[] {$H_5$};
\draw
(H1.-15) edge[auto=right,-]  (H2.195)  
(H2.165) edge[auto=right,-]  (H1.15);  
\draw[-] (H1)--(H2) (H2)--(H3) (H3)--(H4) (H4)--(H5);
%

%

\end{tikzpicture}
\caption{Dynkin graphs for $F_5$ and $H_5$}\label{F5}
\end{figure}

Then the infinite rank module defined by the diagram in Figure
\ref{F5module} proves our claim.

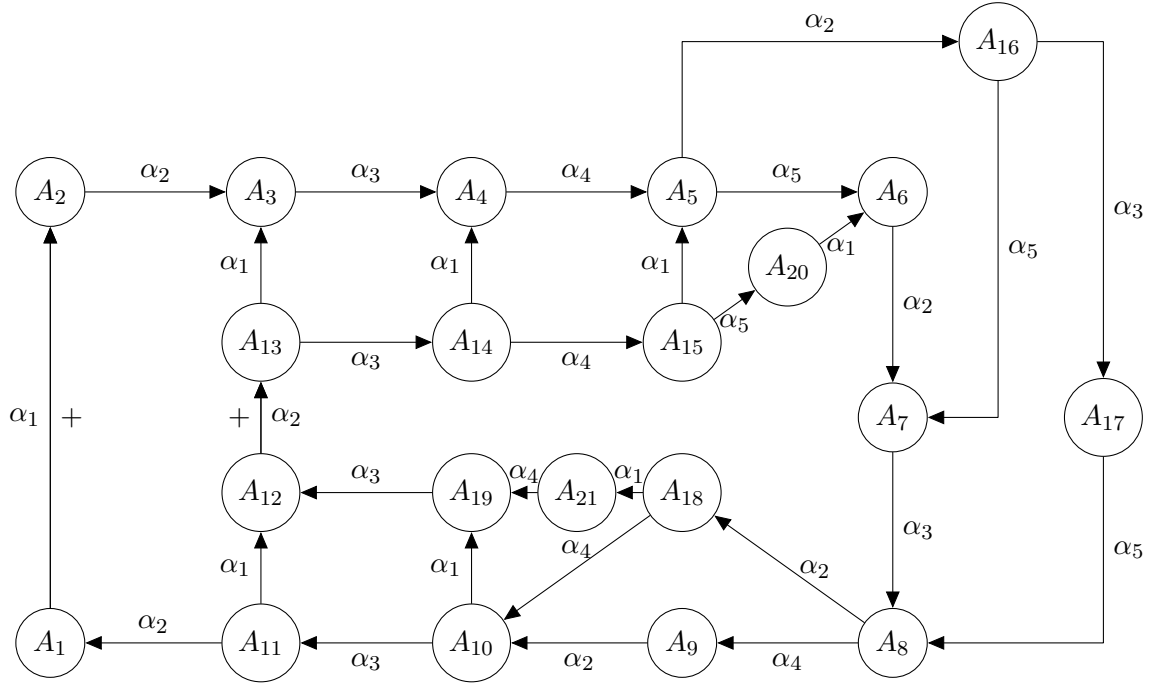
\begin{figure}[ht]
\begin{tikzpicture}[line cap=round,line join=round,>=triangle 45,x=1.4cm,y=1.0cm]
\tikzstyle{every circle node}=[draw]
\modulenode{0}{0}{1}
\modulenode{2}{0}{11}
\modulenode{0}{6}{2}
\modulenode{2}{6}{3}
\modulenode{4}{6}{4}
\modulenode{6}{6}{5}
\modulenode{8}{6}{6}
\modulenode{8}{3}{7}
\modulenode{8}{0}{8}
\modulenode{6}{0}{9}
\modulenode{4}{0}{10}
\modulenode{9}{8}{16}
\modulenode{2}{2}{12}
\modulenode{2}{4}{13}
\modulenode{4}{4}{14}
\modulenode{6}{4}{15}
\modulenode{10}{3}{17}
\modulenode{6}{2}{18}
\modulenode{4}{2}{19}
\modulenode{7}{5}{20}
\modulenode{5}{2}{21}
\moduleedge{1}{2}{$\alpha_1$}{left}
\moduleedge{1}{2}{$+$}{right}
\moduleedge{2}{3}{$\alpha_2$}{above}
\moduleedge{3}{4}{$\alpha_3$}{above}
\moduleedge{4}{5}{$\alpha_4$}{above}
\moduleedge{5}{6}{$\alpha_5$}{above}
\draw [->] (5)--(6,8)--(16) node[midway, above]{$\alpha_2$};
\moduleedge{6}{7}{$\alpha_2$}{right}
\moduleedge{7}{8}{$\alpha_3$}{right}
\moduleedge{8}{9}{$\alpha_4$}{below}
\moduleedge{8}{18}{$\alpha_2$}{right}
\moduleedge{9}{10}{$\alpha_2$}{below}
\moduleedge{10}{11}{$\alpha_3$}{below}
\moduleedge{10}{19}{$\alpha_1$}{left}
\draw[->] (11)--(1) node[midway, above]{$\alpha_2$}; 
\draw[->] (11)--(12) node[midway, left]{$\alpha_1$}; 
\moduleedge{12}{13}{$\alpha_2$}{right}
\moduleedge{12}{13}{$+$}{left}
\moduleedge{13}{3}{$\alpha_1$}{left}
\moduleedge{13}{14}{$\alpha_3$}{below}
\moduleedge{14}{4}{$\alpha_1$}{left}
\moduleedge{14}{15}{$\alpha_4$}{below}
\moduleedge{15}{5}{$\alpha_1$}{left}
\moduleedge{15}{20}{$\alpha_5$}{below}
\draw [->] (16) --(9,3) node[midway,right]{$\alpha_5$} -- (7) ;
\draw [->] (16) -- (10,8) --(17) node[midway,right]{$\alpha_3$};
\draw[->] (17)--(10,0) node[midway,right]{$\alpha_5$}--(8);
\moduleedge{18}{21}{$\alpha_1$}{above}
\moduleedge{18}{10}{$\alpha_4$}{above}
\draw [->] (19) -- (12) node[midway, above]{$\alpha_3$}; 
\moduleedge{20}{6}{$\alpha_1$}{below}
\moduleedge{21}{19}{$\alpha_4$}{above}
\end{tikzpicture}
\caption{An infinite rank module for $F_5$}\label{F5module}
\end{figure}

For the remaining case, it suffices to show a similar module for $H_5$.
This is demonstrated in Figure \ref{H5module}.

\begin{figure}[ht]
\begin{tikzpicture}[line cap=round,line join=round,>=triangle 45,x=1.2cm,y=1.2cm]
\tikzstyle{every circle node}=[draw]
\modulenode{4}{8}{1}
\modulenode{8}{4}{2}
\modulenode{7}{3}{3}
\modulenode{6}{4}{4}
\modulenode{5}{3}{5}
\modulenode{4}{4}{6}
\modulenode{3}{3}{7}
\modulenode{2}{2}{8}
\modulenode{1}{3}{9}
\modulenode{0}{4}{10}
\modulenode{6}{2}{11}
\modulenode{5}{1}{12}
\modulenode{4}{2}{13}
\modulenode{4}{0}{14}
\modulenode{3}{1}{15}
\modulenode{2}{4}{16}
\draw (3.77,7.5) node[anchor=north west] {+};
\moduleedge{1}{2}{$\beta_1$}{above right}
\moduleedge{2}{3}{$\beta_2$}{below right}
\moduleedge{3}{4}{$\beta_1$}{above right}
\moduleedge{3}{11}{$\beta_3$}{below right}
\moduleedge{4}{5}{$\beta_3$}{above left}
\moduleedge{5}{6}{$\beta_2$}{below left}
\moduleedge{5}{13}{$\beta_4$}{below right}
\moduleedge{6}{7}{$\beta_4$}{above left}
\moduleedge{7}{8}{$\beta_5$}{below right}
\moduleedge{7}{16}{$\beta_3$}{below left}
\moduleedge{8}{9}{$\beta_3$}{below left}
\moduleedge{9}{10}{$\beta_4$}{below left}
\draw [->] (10)--(1) node[midway, above left]{$\beta_2$}; 
\moduleedge{11}{5}{$\beta_1$}{above right}
\moduleedge{11}{12}{$\beta_4$}{below right}
\moduleedge{12}{13}{$\beta_1$}{below left}
\moduleedge{12}{14}{$\beta_5$}{below right}
\moduleedge{13}{7}{$\beta_2$}{above right}
\moduleedge{13}{15}{$\beta_5$}{above left}
\moduleedge{14}{15}{$\beta_1$}{below left}
\moduleedge{15}{8}{$\beta_2$}{below left}
\moduleedge{16}{9}{$\beta_5$}{above left}
\end{tikzpicture}
\caption{An infinite rank module for $H_5$}\label{H5module}\end{figure}

Next, we compute the ranks of the finite rank nil-Hecke algebras.
First suppose $W$ is of type $A$ and ${\bf d} = (d,2,\dots,2)$ or
$(2,\dots,2,d)$ for some $d>2$. As all braid relations have length $3$,
$\nilhecke{<4} = \nilhecke{<\infty}$, and so we are done by
Theorem~\ref{Tinfty}.

Otherwise, we fix ${\bf d} = (2,\dots,2)$ henceforth. Now if $W$ is
simply laced (of type $ADE$) or of type $I_2(m)$ with $m<4$, then there
are no braid relations of length $4$ or more, so $\nilhecke{<k} \cong
\nilhecke{<\infty} = NC_W((2,\dots,2))$. This has $\bk$-rank precisely
$|W|$.
If instead $W$ is of type $I_2(m)$, $m\geqslant k=4$, then both sides of
the braid relation (i.e., the longest word in $NC_W((2,2))$) get killed
in $\nilhecke{<4}$, while the remaining $2m-1$ words are not killed in
the quotient. The result now follows by Theorem~\ref{Tdim}.

We tackle the case of $B_n$ next. Let the generators of the corresponding
Weyl group $W(B_n)$ be $s_0, s_1,\ldots, s_{n-1}$, where the labeling is
chosen such that $m_{01}=4$, and $m_{i,i+1}=3$ for all $i\in\{1,\ldots,
n-2\}$. By Theorem \ref{Tdim}, we need to count the set $B_{n,4}$ of
elements $w \in W(B_n)$, for which no reduced expression contains the
substring $s_0s_1s_0s_1$ or $s_1s_0s_1s_0$.
For this, we use the following well-known combinatorial description of
$W(B_n)$ (see, for example, \cite{BB}):

\begin{theorem}\label{Bn}
Let $W$ be the group of all permutations $w$ on $S=\{ \pm 1, \dots, \pm n
\}$ such that $w(-a)=-w(a)\ \forall a \in S$. Then $W = W(B_n)$ is the
Coxeter group of type $B_n$, with generators
\[
s_i=[1,\ldots,i-1,i+1,i,i+2,\ldots,n], \ 1 \leqslant i \leqslant n-1;
\qquad s_0=[-1,2,\ldots,n],
\]
where we write $w=[a_1,a_2,\ldots,a_n] \in W(B_n)$ if $w(i)=a_i\ \forall
1 \leqslant i \leqslant n$.

Further, given any $w\in W$, the length function $\ell(\cdot)$ satisfies
the following properties: 
\begin{enumerate}
\item $\ell(ws_0)=\ell(w)+1$ if and only if $w(1)<0$; and
\item $\ell(ws_i)=\ell(w)+1$ if and only if $w(i)<w(i+1)$ for
$i\in\{1,\ldots, n-1\}$.
\end{enumerate}
\end{theorem}

Using this, we now prove the following lemma:

\begin{lemma}
Let $w=[a_1,\ldots, a_n]\in W(B_n)$ be as above. Then $w\in B_{n,4}$ if
and only the following holds: no pair of indices $i<j$ satisfies
$a_i>a_j$ and $a_i,a_j<0$.
\end{lemma}

\begin{proof}
Suppose $w\in W(B_n)$ does not satisfy the condition: thus, there exist
$i<j$ such that $a_j < a_i < 0$. Call such a pair of indices a
\textit{bad} pair. One can pick $i,j$ with minimal $|i-j|$.
Perform the following operations on $w$: if the leftmost negative number
in $w$ occurs in position $i>1$, multiply by the transposition $s_{i-1}$.
If it occurs in position $1$, multiply by $s_0$. Each move decreases the
length by $1$, and after sufficiently many applications, this brings
$a_i$ to position $1$. By another sequence of length-decreasing
transpositions, we can bring $a_j$ to $2$. Now one verifies that
multiplying by $s_0s_1s_0s_1$ further reduces the length by $4$. Thus
$ww'=w''$, where $w'$ contains the string $s_0s_1s_0s_1$, and
$\ell(w)-\ell(w')=\ell(w'')$. This implies $w'^{-1}w''$ is a reduced
expression for $w$ with the substring $s_1s_0s_1s_0$, and thus $w\not\in
B_{n,4}$.

For the other direction, suppose $w$ satisfies the given condition. We
will prove $w\in B_{n,4}$ by induction on $\ell(w)$. The base cases of
$\ell(w) \leqslant 1$ are clear. For the induction step, suppose to the
contrary that $w$ has a reduced expression with $s_0s_1s_0s_1$ in it, and
let $s_i$ be the rightmost generator in that expression. Then $w=w's_i$,
with $\ell(w')=\ell(w)-1$. There are two cases:

\noindent \textit{Case 1}:
First suppose $s_0s_1s_0s_1$ does not occur at the rightmost position in
this expression for $w$, so it has to occur in the expression for $w'$.
Now $ws_i$ has smaller length than $s$, so if $i\geqslant 1$,
$w(i)>w(i+1)$. Multiplying by $s_i$ swaps these, so if $w$ does not have
a bad pair, swapping $w(i)$ with $w(i+1)$ cannot create further bad
pairs. Thus by the induction hypothesis, $ws_i=w'$ does not have a
reduced expression containing $s_0s_1s_0s_1$, a contradiction. Similarly,
if $i=0$, $w(1)$ must be negative, and swapping its sign cannot create
more bad pairs. By a similar reasoning as above, we are done in this
case.

\noindent \textit{Case 2}:
Otherwise, $w=w''s_0s_1s_0s_1$, with $\ell(w'')=\ell(w)-4$.
As this is a reduced expression, $\ell(ws_1)<\ell(w)$, implying
$w(1)>w(2)$. Now $w_1=ws_1=[w(2),w(1),\ldots]$, so
$\ell(w_1s_0)<\ell(w_1)$ implies $w_1(1)=w(2)<0$. Next,
$w_2=w_1s_0=[-w(2),w(1),\ldots ]$, so $w_3=w_2s_1=[w(1),-w(2),\ldots ]$,
and $w_4=w_3s_0=[-w(1),-w(2),\ldots]$. Since $\ell(w_3s_0)<\ell(s_3)$, we
have $w_3(1)=w(1)<0$. Thus $(1,2)$ is a bad pair in $w$, a contradiction.
This finishes the proof.
\end{proof}

Returning to the main proof, we want to enumerate the elements in
$W(B_n)$ without a bad pair. This is counted in \cite{Sim} as the
$\overline{1}\overline{2}$-avoiding signed permutations, and the answer
is precisely $\sum_{k=0}^n\binom{n}{k}^2k!$. This shows the result for
type $B_n$.

The only remaining cases are $F_4$, $H_3$ and $H_4$: all of finite
$\bk$-rank. By computer checking and use of Theorem \ref{Tdim}, one
verifies the numbers stated in Theorem \ref{T4}. The proof is now
complete.
\end{proof}

We conclude by proving the two corollaries to Theorem~\ref{T4}.

\begin{proof}[Proof of Corollary \ref{T5}]
Since $\nilhecke{<5} \twoheadrightarrow \nilhecke{<4}$, it suffices to
check the cases in Theorem~\ref{T3}. The dihedral case is
easy to verify, and the case of type $A$ with ${\bf d} = (d,2,\dots,2)$
or $(2,\dots,2,d)$ is immediate since there are no braid relations of
length $4$ or more in this case. As the same fact holds for types
$A,B,D,E,F$, hence for these types (with ${\bf d} = (2,\dots,2)$) it
follows that $\nilhecke{<5} = \nilhecke{<\infty} = NC_W((2,\dots,2))$ is
the usual nil-Coxeter algebra, and hence has $\bk$-rank $|W|$.

Finally, suppose $W$ is of type $H_3$ or $H_4$, and ${\bf d} =
(2,\dots,2)$. As all braid words have length either $3$ or $5$, it
follows that $J_{<3} = J_{<4}$ in both cases, and so the algebras and
hence their $\bk$-ranks remain unchanged when passing from $k=4$ to $k=5$.
\end{proof}

\begin{proof}[Proof of Corollary \ref{T6}]
Since $\nilhecke{<k} \twoheadrightarrow \nilhecke{<5}$, to identify the
finite $\bk$-rank nil-Hecke algebras, we need only consider the cases in
Corollary \ref{T5}, i.e., for $k=5$. Once again the case of type $A$ with
${\bf d} = (d,2,\dots,2)$ or $(2,\dots,2,d)$ is immediate, and the
dihedral case of $I_2(m)$ is again easily checked. Since none of the
remaining cases have braid relations of length $6$ or more, in all
non-dihedral cases with ${\bf d} = (2,\dots,2)$, we have $\nilhecke{<k} =
\nilhecke{<\infty} = NC_W((2,\dots,2))$ as above. This completes the
proof.
\end{proof}

\section{Proof of Theorem~\ref{TNil}:
nilpotence of the augmentation ideal}

We next show that the finite-dimensionality of a nil-Hecke algebra is
equivalent to the nilpotence of the augmentation ideal $\m$.

\begin{proof}[Proof of Theorem \ref{TNil}]
The key equivalence to be shown here is $(1) \Longleftrightarrow
(3)$; that $(1) \Longleftrightarrow (2)$ was shown in earlier sections.
First assume $\nilhecke{<k}$ has infinite $\bk$-rank. Recall from above
that every such algebra was shown to have infinite rank by demonstrating
an explicit cyclic $\nilhecke{<k}$-module of infinite $\bk$-rank;
moreover, every diagram above that described such a module has at least
one directed cycle that is ``accompanied'' by a ``$+$'' symbol. Let the
generators along the edges of any such fixed cycle be
$\alpha_1,\alpha_2,\dots, \alpha_m$, and suppose a basis element $v$ in
that module is associated to the initial node of the edge corresponding
to $\alpha_1$. Take an arbitrary positive integer $N$. Then the element
${\bf s} = ({\bf s}_{\alpha_m} \cdots {\bf s}_{\alpha_2} {\bf
s}_{\alpha_1})^j$ belongs to $\m^N$ for large enough $j \gg 0$, and ${\bf
s} v$ is non-zero by how the module action is defined. This implies ${\bf
s} \in \m^N$ is non-zero; and since $N$ was arbitrary, $\m$ is not
nilpotent.

For the other direction, assume $\nilhecke{<k}$ has finite $\bk$-rank.
Applying Theorem~\ref{Tdimplus}, one obtains a finite set $S_{\rm fin}$
of monomials corresponding to equivalence classes in $S$ (as defined in
Theorem~\ref{Tdimplus}) that gives a $\bk$-basis of the algebra in
question. Thus, $S_{\rm fin}$ has an element of maximal length, say
$\ell$. We claim that $\mathfrak m^{\ell+1}=0$, which would prove the
desired nilpotency.

Indeed, consider a monomial ${\bf s}'$ in $\mathfrak m^{\ell+1}$: this
must be a product of strictly more than $\ell$ generators. If ${\bf s}'$
equals some monomial composed of $\ell$ generators or less, one can
reduce this via using the defining relations of $\nilhecke{<k}$. Since
applying the braid relations does not change the number of monomials,
this can be reduced to some monomial containing the string ${\bf
s}_i^{d_i}$ or a braid word of length $\geqslant k$ at some point, and
hence it is zero in $\nilhecke{<k}$.

On the other hand, if ${\bf s}'$ is not equal in $\nilhecke{<k}$ to any
monomial composed of $\ell$ generators or less, the corresponding string
of ${\bf s}_i$'s cannot belong to any equivalence class in $S$, hence
reduces to $0$ (as in the proof of Theorem~\ref{Tdimplus}).

This completes the proof of the equivalence.
Finally, suppose these conditions hold and $\nilhecke{<k}$ is
finite-dimensional, where $\bk$ is now a field. It is clear that if $x
\in \m$ and $\m^{\ell+1} = 0$, then $(1-x)^{-1} = 1 + x + \cdots +
x^\ell$, and so $\nilhecke{<k}$ is local.
\end{proof}

\begin{remark}
In light of Theorem~\ref{Tcomplex} (proved in the next section), one can
ask if Theorem~\ref{TNil} holds for the finite-rank nil-Hecke algebras
over complex reflection groups. This is indeed the case, because -- as we
explain in the next section -- there are precisely two such algebras, and
they essentially arise from real affine reflection groups, so
Theorem~\ref{TNil} applies to them too.
\end{remark}

\section{Proof of Theorem~\ref{Tcomplex}: nil-Hecke algebras over complex
reflection groups}

We now turn to the assertion -- extending Marin's assertion \cite{Ma} for
nil-Coxeter algebras and its extension in \cite{Kh2} to generalized
nil-Coxeter algebras -- that finite complex reflection groups admit no
finite-dimensional nil-Hecke algebras.
This is ``half'' of Theorem~\ref{Tcomplex}, and we then produce (exactly)
two finite-dimensional examples over infinite complex groups. After the
proof, we turn to the remaining case of ``non-genuine'' crystallographic
groups $W$.

\begin{proof}[Proof of Theorem~\ref{Tcomplex}]
We begin by discussing the finite (irreducible) complex reflection
groups, for which a presentation can be found in \cite{BMR2}; note that
$k \geq 3$ here. Now a straightforward verification using this
presentation reveals that all corresponding nil-Hecke algebras of the
form $\nilhecke{<k}$ have infinite $\bk$-rank. Indeed, the arguments in
\cite[Section 6, Cases 10--12]{Kh2} (which prove the same statement in
the special case $k=\infty$) work for $3 \leq k < \infty$ as well. The
only point where our proof diverges from \cite{Kh2} is when $W=G_{29}$.

In this exceptional case, we note that the corresponding (``usual'')
nil-Coxeter algebra is generated by the generators ${\bf s}_s, {\bf s}_t,
{\bf s}_u, {\bf s}_v$ subject to the following relations:
\begin{gather*} 
{\bf s}_s^2={\bf s}_t^2={\bf s}_u^2={\bf s}_v^2=0,\quad {\bf s}_s{\bf
s}_v={\bf s}_v{\bf s}_s,\quad {\bf s}_s{\bf s}_u={\bf s}_u{\bf s}_s,\\ 
{\bf s}_s{\bf s}_t{\bf s}_s={\bf s}_t{\bf s}_s{\bf s}_t,\quad {\bf
s}_v{\bf s}_t{\bf s}_v={\bf s}_t{\bf s}_v{\bf s}_t,\quad {\bf s}_u{\bf
s}_v{\bf s}_u={\bf s}_v{\bf s}_u{\bf s}_v,\\
{\bf s}_t{\bf s}_u{\bf s}_t{\bf s}_u={\bf s}_u{\bf s}_t{\bf s}_u{\bf
s}_t,\quad {\bf s}_v{\bf s}_t{\bf s}_u{\bf s}_v{\bf s}_t{\bf s}_u={\bf
s}_t{\bf s}_u{\bf s}_v{\bf s}_t{\bf s}_u{\bf s}_v.
\end{gather*}

The corresponding nil-Hecke algebras $\nilhecke{<k}$ are obtaining as a
quotient of this by (first replacing the Coxeter relations/exponents $2$
by $d_i$, and then) killing suitable braid words depending on $k$. One
notes that for all of these, the diagram in Figure \ref{G29module}
defines a suitable infinite-rank $\bk$-module. This completes the proof
for finite complex reflection groups.

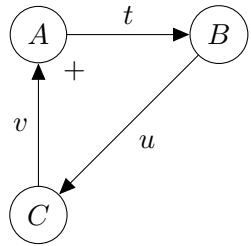
\begin{figure}[ht]
\begin{tikzpicture}[line cap=round,line join=round,>=triangle 45,x=1.2cm,y=1.2cm]
\tikzstyle{every circle node}=[draw]
\modulenodeA{0}{0}{A};
\modulenodeA{2}{0}{B};
\modulenodeA{0}{-2}{C};
\draw (.4,-.4) node {+};
\moduleedge{A}{B}{$t$}{above};
\moduleedge{B}{C}{$u$}{below right};
\moduleedge{C}{A}{$v$}{left}
\end{tikzpicture}
\caption{An infinite rank module for $G_{29}$}\label{G29module}
\end{figure}

We next turn to the infinite irreducible groups (i.e., associated to
connected braid diagrams). These were described in \cite{Mal,Po1}, and
are each associated to a complex affine space $E$ with group of
translations $V$. We fix a basepoint $v_0 \in E$ and identify $GL(V)
\ltimes V \cong {\rm Aff}(E)$, the affine transformation group of $E$. As
is explained in \textit{loc.\ cit.}, $W \leq {\rm Aff}(E)$. Now invoke
the results on \cite[pp.~30]{Po1}, to note that if $\nilhecke{<k}$ has
finite rank, then the same holds with ${\bf d}$ replaced by
$(2,2,\dots,2)$. But then $W = W_{\mathbb{R}}$ is an irreducible affine
Coxeter group that acts on a real form $E_\mathbb{R}$ of $E$, and the
action of $W$ on $E = \mathbb{C} \otimes_\mathbb{R} E_\mathbb{R}$ is
simply its complexification. Now $\nilhecke{<k} =
n\mathcal{H}(W_{\mathbb{R}}, {\bf d}, J_{<k})$ has finite $\bk$-rank,
hence so does $n\mathcal{H}(W_{\mathbb{R}}, (2,2,\dots,2), J_{<k})$. As
$W_\mathbb{R}$ is affine and $k \geq 3$, our results in previous sections
(see the penultimate row of Table~\ref{Ttable}) show that $k=3$ and
$W = W_\mathbb{R}$ is affine of type $E_9$ or $F_5$.

The remaining case is when $W$ is a genuine crystallographic group, i.e.,
$E/W$ is compact (with $W \leq {\rm Aff}(E)$) and $W /(W \cap V)$ is not
the complexification of a real reflection group (where $V \leq {\rm
Aff}(E)$ consists of the translations). For such groups, Coxeter-type
presentations were provided by Malle in \cite[Tables  I, II]{Mal}; and in
Case~$14$ in the proof of \cite[Theorem D]{Kh2}, it was shown that there
were no generalized nil-Coxeter algebras (i.e., with arbitrary ${\bf d}$
but with $k=\infty$). One can now verify that the proofs in \cite{Kh2} in
\textit{loc.\ cit.}\ for the $k = \infty$ case, also go through verbatim
if $3 \leq k < \infty$.
\end{proof}

\subsection{Non-genuine, crystallographic complex reflection groups}

As explained in \cite{Po1}, there is a third class of (irreducible)
discrete infinite complex reflection groups, which were termed
``non-genuine crystallographic'' in \cite[Section~6]{Kh2} -- and so it is
natural to ask if the nil-Hecke algebras over these groups also have
infinite $\bk$-rank. As we now explain, a ``specific'' Coxeter-type
presentation for these groups -- arising out of work of Popov \cite{Po1}
and Ion--Sahi \cite{IS} -- involves a braid-type relation equating two
braid words of unequal lengths, and so the nil-Hecke algebras here are
not defined.

We now elaborate on this. First, like the other complex groups, the
non-genuine crystallographic complex reflection groups also admit a
Coxeter-type presentation, as was explained to the second author by Popov
\cite{Po2}. For the details, we refer the reader to Case~$15$ in the
proof of \cite[Theorem~D]{Kh2}. In brief: retaining the notation in the
preceding proof, $W \leq {\rm Aff}(E) = GL(V) \ltimes V$ is now given by
$W = {\rm Lin}(W) \ltimes {\rm Tran}(W)$, where ${\rm Lin}(W) := W \cap
V$ and ${\rm Tran}(W) := W / (W \cap V)$. Moreover, $W' := {\rm Lin}(W)$
is a finite irreducible real reflection group (in fact a Weyl group), and
${\rm Tran}(W) \cong \Lambda_1 \oplus \tau \Lambda_1$ for some
$\Z$-lattice $\Lambda_1$ of full rank, and a scalar $\tau \in \mathbb{C}
\setminus \mathbb{R}$ (which we may thus take to lie in the upper
half-plane). Moreover, the affine Weyl group $\widetilde{W}$ over $W$
satisfies:
\[
W' \ltimes \Lambda_1 \cong \widetilde{W} \cong W' \ltimes \tau \Lambda_1,
\]
so that $W$ is, in a sense, a ``double affine Weyl group''.

Now fix a simple base $\Pi$ for the root system $\Phi$ of $W'$, and let
$\theta$ denote the unique highest root for $\Phi$. Define $a_{01} :=
t(\theta) s_\theta$ to be the extra affine reflection in $W' \ltimes
\Lambda_1$, where $t(\theta)$ is the translation by $\theta$ in
$E_\mathbb{R}$ and $s_\theta \in O(E_\mathbb{R})$ is the reflection
sending $\theta$ to $-\theta$. Similarly, let $a_{02} := t(\tau \theta)
s_\theta$ denote the extra affine reflection in $W' \ltimes \tau
\Lambda_1$. Then the non-genuine crystallographic infinite irreducible
discrete complex reflection group $W$ (which is a double affine Weyl
group) is generated by $\{ s_\alpha : \alpha \in \Pi \}$ and $a_{01},
a_{02}$. This does have a Coxeter-type presentation, as was explained to
the second author by Sahi \cite{Sahi}. To obtain this, first note that
the map
\[
a_{03} := a_{01} s_\theta a_{02} = t((1+\tau)\theta) s_\theta
\]
also squares to the identity map on $E_\mathbb{R}$. In particular, we
obtain the further relation
\begin{equation}\label{Esahi}
a_{01} a_{03} a_{02} = s_\theta;
\end{equation}
moreover, $W'$ and $a_{03}$ also generate an affine Weyl group, again
isomorphic to $\widetilde{W}$.

Now Ion and Sahi have shown -- see the end of \cite[Section 5]{IS} --
that $W$ has a Coxeter-type presentation, with generators given by $\{
s_\alpha : \alpha \in \Pi \} \sqcup \{ a_{01}, a_{02}, a_{03} \}$; the
relations are that the subset $\{ s_\alpha : \alpha \in \Pi \} \sqcup \{
a_{0j} \}$ satisfies the Coxeter presentation for $\widetilde{W}$ for
$j=1,2,3$, and the additional relation~\eqref{Esahi} holds. (Note that
$a_{0j} a_{0j'}$ has infinite order in $W$ for $j \neq j'$.)

With this presentation at hand, we return to our original question of
interest: to examine the finite dimensionality (or $\bk$-rank) of the
associated nil-Hecke algebra $\nilhecke{<k}$ for $k \geq 3$. Note that if
this happens then the same holds with ${\bf d}$ replaced by
$(2,\dots,2)$, and moreover with $W$ replaced by $\widetilde{W}$ (by
killing $a_{02}, a_{03}$). As in the proof of Theorem~\ref{Tcomplex},
this implies that $k=3$ and $W'$ is finite of type $E_8$ or $F_4$. (In
particular, there is a unique node $\alpha \in \Pi$ to which the three
affine nodes $a_{0j}$ are attached, each by a single edge.) But now the
braid words on both sides of the extra relation~\eqref{Esahi} have
unequal lengths, and so the algebra $\nilhecke{<3}$ is not defined.

\section{Proof of Theorem~\ref{TFrobenius}: classification of Frobenius
nil-Hecke algebras}

The goal of this final section is to show Theorem~\ref{TFrobenius}. We
\textbf{assume} throughout this section that $\bk$ is a field. Below, we
will use without further mention the fact that the spaces of
left-primitive and right-primitive elements of $\nilhecke{<k}$ are
linearly isomorphic -- working over any Coxeter group $W$ -- via the
anti-involution on $\nilhecke{<k}$ that sends each generator
$\mathbf{s}_i$ to itself.
As a first step, we classify the cases where the set of primitive
elements have dimension $1$.

\begin{theorem}\label{TPrim}
Fix a Coxeter group $W$ with related data $I,J,{\bf S},\R$, integers $d_i
\geqslant 2\ \forall i$ and $1 \leqslant k \leqslant \infty$. 
Suppose the corresponding nil-Hecke $\bk$-algebra $\nilhecke{<k}$ is
finite-dimensional. Then the space of primitive elements in
$\nilhecke{<k}$ is one-dimensional if and only if one of the following
holds:
\begin{enumerate}
    \item $d_i=2$ for all $i$, and $W$ contains no braid relations of length $k$ or more;
    \item $k\in\{1,2\}$, $W$ is of type $B$ or $A_1$, and $d_i=2$ for all $i$;
    \item $k\in\{1,2\}$, $W$ is of type $A$, $d_i=3$ for for exactly one of the pendant vertices of the Coxeter graph of $W$ and $d_i=2$ for all other vertices;
    \item $k=3$, $W$ is of the type $A_1$ or $H_3$ and $d_i=2$ for all $i$;
    \item $k=3$, $W$ is of type $A_2$ with $d_i=3$ for exactly one vertex and $d_i=2$ for the other;
    \item $k\geqslant 3$ and $W$ is of type $A_1$ with $d_i>2$ for the only vertex present.
\end{enumerate}
\end{theorem}

\begin{proof}
We first show that the cases other than those mentioned above yield at
least two linearly independent primitive elements -- in fact, two
distinct primitive monomials.\medskip

\noindent \textit{Part 1}:
Suppose $k\in\{1,2\}$. Suppose first the associated Coxeter graph is a
simply laced tree with at least two vertices, and no vertex $i$ has
$d_i>2$. This tree necessarily has a maximal path of length at least two,
say
\[
v_1\quad\longleftrightarrow \quad v_2 \quad\longleftrightarrow\quad\cdots
\quad\longleftrightarrow \quad v_n.
\]
Letting $\mathbf{s}_i$ be the generator associated to vertex $v_i$, one
can check that $\mathbf{s}_1\mathbf{s}_2\cdots\mathbf{s}_n$ and
$\mathbf{s}_n\mathbf{s}_{n-1}\cdots \mathbf{s}_1$ are two distinct
primitive monomials. In case the graph contains a multiple edge
(respectively, a vertex with $d_i>2$) that is not at a pendant vertex,
one can again form a maximal path containing this edge (respectively,
this vertex) at some point other than the two ends. Then the same
argument as above gives again two distinct primitive monomials.

Now suppose this graph has a multiple edge at its end, with all $d_i=2$.
If there is some vertex $v$ with degree at least $3$, then deleting this
gives rise to three or more disconnected components, only one of which
contains the multiple edge. Suppose $v$ is connected to $v_1$ and $v_2$,
each of which belong to a component not containing the multiple edge.
Extending the path $v_1 v v_2$ gives us a maximal path with no multiple
edge, and as before, we can conclude this does not correspond to a
Frobenius algebra. 

Thus we can only have a path graph with a multiple edge as one end. Label
the vertices $v_1,v_2,\ldots, v_n$ where $m=m_{v_{n-1},v_n}>3$. Since
this is not of type $B$, $m\ge 5$. If $m=5$, again calling the generator
associated to vertex $v_i$ as $\mathbf{s}_i$, we see that
$\mathbf{s}_1\mathbf{s}_2\cdots
\mathbf{s}_{n-1}\mathbf{s}_n\mathbf{s}_{n-1}\ldots\mathbf{s}_2\mathbf{s}_1$
and $\mathbf{s}_1\mathbf{s}_2\cdots
\mathbf{s}_{n-1}\mathbf{s}_n\mathbf{s}_{n-1}\mathbf{s}_n$ are distinct
and primitive. If $m\ge 6$, then
$\mathbf{s}_1\mathbf{s}_2\cdots
\mathbf{s}_{n-1}\mathbf{s}_n\mathbf{s}_{n-1}\ldots\mathbf{s}_2\mathbf{s}_1
\neq \mathbf{s}_1\mathbf{s}_2\cdots
\mathbf{s}_{n-1}\mathbf{s}_n\mathbf{s}_{n-1}\mathbf{s}_n\mathbf{s}_{n-1}\mathbf{s}_{n-2}\ldots\mathbf{s}_2\mathbf{s}_1$
are both primitive.

According to Theorem \ref{T12} and the previous analysis, the only
remaining case is when the Coxeter graph is a simply laced tree, with
$d_v>2$ for a pendant vertex $v$. Using the above arguments, one can show
that this graph is necessarily a path graph. Serially label its vertices
$v_1,v_2, \ldots, v_{n-1},v_n=v$.
If $d_v>3$, then again $\mathbf{s}_1\mathbf{s}_2\cdots
\mathbf{s}_{n-1}\mathbf{s}_n^2\mathbf{s}_{n-1}\ldots\mathbf{s}_2\mathbf{s}_1
\neq \mathbf{s}_1\mathbf{s}_2\cdots
\mathbf{s}_{n-1}\mathbf{s}_n^3\mathbf{s}_{n-1}\ldots\mathbf{s}_2\mathbf{s}_1$
are two distinct primitive monomials. Thus the only remaining possibility
is $d_v=3$, which is accounted for in (3).\medskip

\noindent \textit{Part 2}:
Next, assume $k=3$. For each case where $\nilhecke{<k}$ is
finite-dimensional as described in Theorem \ref{T3}, we exhibit two
distinct primitive monomials, except for the cases mentioned in the
statement of Theorem \ref{TPrim}. These monomials are tabulated in Table
\ref{TPrimTable}.\medskip

\begin{table}[ht]
    \centering
    \begin{tabular}{|c|c| >{\centering\arraybackslash}m{7cm}|}
    \hline
         Type & Dynkin diagram & Primitive monomials  \\
         \hline
          $A_n (n\geqslant 2)$ & \dynkin[labels={1,2,n}] A{oo.o} & 
              $\mathbf{s}_1\mathbf{s}_2\cdots\mathbf{s}_n$, \quad $\mathbf{s}_n\cdots \mathbf{s}_2\mathbf{s}_1$\\
              \hline
          $B_n (n\geqslant 3)$ & \dynkin [backwards,labels={n,3,2,1},arrows=false] B{o.ooo} &  $\mathbf{s}_n\mathbf{s}_{n-1}\cdots \mathbf{s}_2\mathbf{s}_1\mathbf{s}_2\cdots\mathbf{s}_{n-1}\mathbf{s}_n$, $\mathbf{s}_{n-1}\cdots \mathbf{s}_2\mathbf{s}_1\mathbf{s}_n\mathbf{s}_2\cdots\mathbf{s}_{n-1}$\\
           \hline
          $D_{n+1} (n\geqslant 3)$ & \dynkin[labels={1,2,,n-1,n,n'}, label directions={,,,right,,}] D{oo.oooo} & $\mathbf{s}_1\mathbf{s}_2\cdots\mathbf{s}_{n-1}\mathbf{s}_n\mathbf{s}_{n'}\mathbf{s}_{n-1}\cdots\mathbf{s}_2\mathbf{s}_1$, $\mathbf{s}_2\cdots\mathbf{s}_{n-1}\mathbf{s}_n\mathbf{s}_1\mathbf{s}_{n'}\mathbf{s}_{n-1}\cdots\mathbf{s}_2$\\
          \hline
          $E_6$ & \dynkin[labels={1,6,2,3,4,5}] E{oooooo}& $\mathbf{s}_5\mathbf{s}_4\mathbf{s}_3\mathbf{s}_2\mathbf{s}_6\mathbf{s}_3\mathbf{s}_1\mathbf{s}_6\mathbf{s}_2\mathbf{s}_3\mathbf{s}_4\mathbf{s}_5$, $\mathbf{s}_1\mathbf{s}_2\mathbf{s}_3\mathbf{s}_4\mathbf{s}_6\mathbf{s}_3\mathbf{s}_5\mathbf{s}_6\mathbf{s}_4\mathbf{s}_3\mathbf{s}_2\mathbf{s}_1$\\
          \hline
          $E_n (n\geqslant 7)$ & \dynkin[labels={1,n,2,3,4,,n-1}] E{ooooo.oo} & $\mathbf{s}_{n-1}\cdots \mathbf{s}_4\mathbf{s}_3\mathbf{s}_2\mathbf{s}_7\mathbf{s}_1\mathbf{s}_3\mathbf{s}_7\mathbf{s}_2\mathbf{s}_3\mathbf{s}_4\cdots\mathbf{s}_{n-1}$, $\mathbf{s}_{n-2}\cdots \mathbf{s}_4\mathbf{s}_3\mathbf{s}_2\mathbf{s}_7\mathbf{s}_1\mathbf{s}_{n-1}\mathbf{s}_3\mathbf{s}_7\mathbf{s}_2\mathbf{s}_3\mathbf{s}_4\cdots\mathbf{s}_{n-2}$\\
          \hline
          $F_4$ & \dynkin[labels={1,2,3,4},arrows=false] F{oooo} &
	  $\mathbf{s}_4\mathbf{s}_3\mathbf{s}_2\mathbf{s}_1\mathbf{s}_3\mathbf{s}_2\mathbf{s}_3\mathbf{s}_4$, \quad $\mathbf{s}_1\mathbf{s}_2\mathbf{s}_3\mathbf{s}_4\mathbf{s}_2\mathbf{s}_3\mathbf{s}_2\mathbf{s}_1$\\
          \hline
          $F_n (n\geqslant 5)$ & \dynkin[extended, labels={n,,3,2,1},backwards, arrows=false] F{o.ooo} & $\mathbf{s}_n\cdots\mathbf{s}_3\mathbf{s}_2\mathbf{s}_1\mathbf{s}_3\mathbf{s}_2\mathbf{s}_3\cdots\mathbf{s}_n$, $\mathbf{s}_{n-1}\cdots\mathbf{s}_3\mathbf{s}_2\mathbf{s}_1\mathbf{s}_n\mathbf{s}_3\mathbf{s}_2\mathbf{s}_3\cdots\mathbf{s}_{n-1}$\\
          \hline
          $H_4$ & \dynkin[labels={1,2,3,4}] H{oooo}& $\mathbf{s}_4\mathbf{s}_3\mathbf{s}_2\mathbf{s}_1\mathbf{s}_2\mathbf{s}_1\mathbf{s}_3\mathbf{s}_2\mathbf{s}_1\mathbf{s}_2\mathbf{s}_4\mathbf{s}_3$, $\mathbf{s}_4\mathbf{s}_3\mathbf{s}_2\mathbf{s}_1\mathbf{s}_2\mathbf{s}_1\mathbf{s}_3\mathbf{s}_2\mathbf{s}_1\mathbf{s}_2\mathbf{s}_3\mathbf{s}_4$\\
          \hline
          $H_n (n\geqslant 5)$ &\dynkin[labels={1,2,3,n}] H{ooo.o} & $\mathbf{s}_n\cdots\mathbf{s}_3\mathbf{s}_2\mathbf{s}_1\mathbf{s}_2\mathbf{s}_1\mathbf{s}_3\mathbf{s}_2\mathbf{s}_1\mathbf{s}_2\mathbf{s}_3\cdots\mathbf{s}_n$, $\mathbf{s}_{n-1}\cdots\mathbf{s}_3\mathbf{s}_2\mathbf{s}_1\mathbf{s}_2\mathbf{s}_1\mathbf{s}_n\mathbf{s}_3\mathbf{s}_2\mathbf{s}_1\mathbf{s}_2\mathbf{s}_3\cdots\mathbf{s}_{n-1}$\\
          \hline
          $I_2(m) (m\geqslant 4)$ & \dynkin[labels={1,2},gonality=m]I{oo}  & $\;\;\mathbf{s}_1\mathbf{s}_2\mathbf{s}_1\cdots$ ($m-1$ generators),\;\; $\mathbf{s}_2\mathbf{s}_1\mathbf{s}_2\cdots$ ($m-1$ generators)\\
          \hline
          $A_n(n\geqslant 3),$ & \multirow{2}{*}{\dynkin[labels={1,2,n},labels*={d,,}] A{*o.o}}
	  & $\mathbf{s}_n\cdots \mathbf{s}_2\mathbf{s}_1\mathbf{s}_1\mathbf{s}_2\cdots\mathbf{s}_n$,\\
          $\mathbf{d}=(d,2,\ldots, 2), d\geqslant3$ & &$\mathbf{s}_{n-1}\cdots\mathbf{s}_2\mathbf{s}_1\mathbf{s}_n\mathbf{s}_1\mathbf{s}_2\mathbf{s}_{n-1}$\\
          \hline
          $A_2,\mathbf d=(d,2),$ & \multirow{2}{*}{\dynkin[labels={1,2},labels*={d,}] A{*o}} & $\mathbf{s}_2\mathbf{s}_1\mathbf{s}_1\mathbf{s}_2$,\\
          $ d>3$ & &$\mathbf{s}_2\mathbf{s}_1\mathbf{s}_1\mathbf{s}_1\mathbf{s}_2$\\
          \hline
    \end{tabular}
    \caption{Primitive monomials for $k=3$}
    \label{TPrimTable}
\end{table}

\noindent \textit{Part 3}:
Now let $k=4$. For cases where the Coxeter graph has no edges of label
$4$ or more, the corresponding algebra $\nilhecke{<4}$ is identical to
$\nilhecke{<\infty}$, which has been analyzed in \cite{Kh2}. The only
remaining cases are $B_n$, $F_4$, $H_3$ and $H_4$ and $I_2(m)$ for
$m\geqslant 4$.

For $B_n$ with $n\geqslant 2$, using the notation in Theorem \ref{Bn},
we see that the signed permutations
\[
w' := [-n,-(n-1),\ldots,-1], \ \
w := [n-1,n-2,\ldots, 2,1,-n] \ \in \, W
\]
both correspond to right-primitive monomials $\mathbf{s}(w'),
\mathbf{s}(w) \in \nilhecke{<4}$ (i.e., $x\mathfrak m=0$ for $x =
\mathbf{s}(w'), \mathbf{s}(w)$). Further, one can check that $w',w$ are
self-inverses, which easily implies that $\mathbf{s}(w'), \mathbf{s}(w)$
are left-primitive as well (i.e., $\mathfrak m x=0$), as desired. We
tabulate the results for the remaining cases in Table
\ref{TPrimTable4}.\medskip

\begin{table}[ht]
    \centering
    \begin{tabular}{|c|c| >{\centering\arraybackslash}m{9cm}|}
    \hline
         Type & Dynkin diagram & Primitive monomials  \\
         \hline
          $F_4$ & \dynkin[labels={1,2,3,4},arrows=false] F{oooo} &
	  $\mathbf{s}_1\mathbf{s}_2\mathbf{s}_3\mathbf{s}_4\mathbf{s}_1\mathbf{s}_2\mathbf{s}_3\mathbf{s}_1\mathbf{s}_2\mathbf{s}_1$, \quad $\mathbf{s}_3\mathbf{s}_4\mathbf{s}_3\mathbf{s}_2\mathbf{s}_3\mathbf{s}_1\mathbf{s}_2\mathbf{s}_3\mathbf{s}_4\mathbf{s}_3$\\
          \hline
          $H_3$ & \dynkin[labels={1,2,3}] H{ooo}&
	  $\mathbf{s}_3\mathbf{s}_2\mathbf{s}_1\mathbf{s}_2\mathbf{s}_1\mathbf{s}_3\mathbf{s}_2\mathbf{s}_1\mathbf{s}_2\mathbf{s}_3$, \quad $\mathbf{s}_1\mathbf{s}_2\mathbf{s}_1\mathbf{s}_3\mathbf{s}_2\mathbf{s}_1\mathbf{s}_3\mathbf{s}_2\mathbf{s}_1\mathbf{s}_3\mathbf{s}_2\mathbf{s}_1$\\
          \hline
          $H_4$ & \dynkin[labels={1,2,3,4}] H{oooo}& $\mathbf{s}_4\mathbf{s}_3\mathbf{s}_2\mathbf{s}_1\mathbf{s}_2\mathbf{s}_1\mathbf{s}_3\mathbf{s}_2\mathbf{s}_1\mathbf{s}_4\mathbf{s}_3\mathbf{s}_2\mathbf{s}_1\mathbf{s}_3\mathbf{s}_2\mathbf{s}_1\mathbf{s}_3\mathbf{s}_2\mathbf{s}_1\mathbf{s}_4\mathbf{s}_3\mathbf{s}_2\mathbf{s}_1\mathbf{s}_2\mathbf{s}_3\mathbf{s}_4$, $\mathbf{s}_1\mathbf{s}_2\mathbf{s}_1\mathbf{s}_3\mathbf{s}_2\mathbf{s}_1\mathbf{s}_4\mathbf{s}_3\mathbf{s}_2\mathbf{s}_1\mathbf{s}_4\mathbf{s}_3\mathbf{s}_2\mathbf{s}_1\mathbf{s}_4\mathbf{s}_3\mathbf{s}_2\mathbf{s}_1\mathbf{s}_4\mathbf{s}_3\mathbf{s}_2\mathbf{s}_1$\\
          \hline
           $I_2(m)$ & \multirow{2}{*}{\dynkin[labels={1,2},gonality=m]I{oo}}  & $\mathbf{s}_1\mathbf{s}_2\mathbf{s}_1\cdots$ ($m-1$ generators), \\
           $(m\geqslant 4)$ & &$\mathbf{s}_2\mathbf{s}_1\mathbf{s}_2\cdots$ ($m-1$ generators)\\
          \hline
    \end{tabular}
    \caption{Primitive monomials for $k=4$}
    \label{TPrimTable4}
\end{table}

\noindent \textit{Part 4}:
First suppose $k=5$. For $I_2(m)$, a similar argument as before applies.
The cases of $H_3$ and $H_4$ are identical to those for $k=4$, and the
remaining cases are identical to those for $k=\infty$.
Next, for $6\leqslant k\leqslant \infty$, again the proof for $I_2(m)$
follows along similar lines as above, and the remaining cases are
identical to their corresponding $k=\infty$ analogues.\medskip

\noindent \textit{Part 5}:
It remains to prove the cases listed in Theorem \ref{TPrim} indeed yield
one-dimensional spaces of primitive elements. We consider each of the
cases separately: 
\begin{enumerate}
    \item In this case, $\nilhecke{<k}=\nilhecke{<\infty}$, and this is
    analyzed in \cite{Kh2}.
    \item The case where $W$ has type $A_1$ is easy to verify explicitly.
    Now suppose $W$ has type $B$: say it contains the vertices $v_1,
    v_2,\ldots, v_n$ on a path in that order, and $m_{v_{n-1},v_n}=4$.
    Then according to the results in \cite{Ha} and Theorem
    \ref{Tdimplus}, the non-zero monomials have one of the following
    forms:
    \begin{enumerate}
    \item[(i)] $\mathbf{s}_i\mathbf{s}_{i+1}\cdots \mathbf{s}_j$ or
    $\mathbf{s}_j\mathbf{s}_{j-1}\cdots \mathbf{s}_i$ for $1\leqslant
    i\leqslant j\leqslant n$. The first monomial can be right-multiplied
    by $\mathbf{s}_{j+1}$ if $j<n$ and by $\mathbf{s}_{n-1}$ if $j=n$
    without yielding zero, and thus is not primitive. A similar logic
    holds for monomials of the second form.
    \item[(ii)] $\mathbf{s}_i\mathbf{s}_{i+1}\cdots
    \mathbf{s}_{n-1}\mathbf{s}_n\mathbf{s}_{n-1}\cdots
    \mathbf{s}_{j+1}\mathbf{s}_j$. This can be left-multiplied by
    $s_{i-1}$ if $i>1$ and can be right-multiplied by $\mathbf{s}_{j-1}$
    if $j>1$, so the only case where this is primitive is $i=j=1$, and
    here it indeed yields a primitive monomial. 
    \end{enumerate}
    Since we have considered all possible monomials, this completes the
    proof.
    \item Again, in this case, we can list out the possible monomials.
    Suppose the vertices are $v_1,v_2,\ldots, v_n$ on a path, with
    $d_{v_n}=3$. The possible monomials are of the form:
    \begin{enumerate}
    \item[(i)] $\mathbf{s}_i\mathbf{s}_{i+1}\cdots \mathbf{s}_j$ or
    $\mathbf{s}_j\mathbf{s}_{j-1}\cdots \mathbf{s}_i$ for $1\leqslant
    i\leqslant j\leqslant n$. These can be eliminated as before.
    \item[(ii)] $\mathbf{s}_i\mathbf{s}_{i+1}\cdots
    \mathbf{s}_{n-1}\mathbf{s}_n^2\mathbf{s}_{n-1}\cdots
    \mathbf{s}_{j+1}\mathbf{s}_j$. As before, the only case where
    this is primitive is $i=j=1$, and that proves the claim.
    \end{enumerate}
    \item and (5) are both small finite cases. The case of $H_3$ can be
    verified computationally, while the other two are simple enough for
    manual verification.
    \item[(6)] Here, the algebra is simply $\bk[\mathbf s]/\langle
    \mathbf s^{d_i}\rangle$, which has only one primitive monomial:
    $\mathbf s^{d_i-1}$.
\end{enumerate}
This concludes our proof.
\end{proof}

Next, we enumerate the cases where $\nilhecke{<k}$ has a one-dimensional
space of right-primitive elements.

\begin{theorem}\label{TRightPrim}
Fix a Coxeter group $W$ with related data $I,J,{\bf S},\R$, integers $d_i
\geqslant 2\ \forall i$ and $1 \leqslant k \leqslant \infty$.
Suppose the corresponding nil-Hecke $\bk$-algebra $\nilhecke{<k}$ is
finite-dimensional. Then the set of right-primitive elements in
$\nilhecke{<k}$ has dimension one if and only if one of the following
holds:
\begin{enumerate}
\item $d_i=2$ for all $i$, and $W$ is a finite Coxeter group that
contains no braid relations of length $k$ or more;
\item $W$ is of type $A_1$.
\end{enumerate}
\end{theorem}

\begin{proof}
Clearly if the space of right-primitive elements is one-dimensional, so
is its subspace of primitive elements. Thus we only need to consider the
cases listed in Theorem \ref{TPrim}. We check each of them separately:

\begin{enumerate}
\item In this case, we do have a one-dimensional set of right-primitive
elements, as seen in \cite{Kh2}.

\item If $W$ is of type $B_n$ with $n \geq 3$, label the vertices as in
the proof of Theorem \ref{TPrim}, part~5. Then
$\mathbf{s}_2\mathbf{s}_1 \neq \mathbf{s}_3\mathbf{s}_2\mathbf{s}_1$ are
two right-primitive monomials. On the other hand, if $n=2$, then
$\mathbf{s}_1\mathbf{s}_2\mathbf{s}_1$ and
$\mathbf{s}_2\mathbf{s}_1\mathbf{s}_2$ are two distinct right-primitive
monomials. This leaves us with the case of $A_1$, which can be checked to
have a one-dimensional space of right-primitive elements as at the end of
the proof of Theorem \ref{TPrim}.

\item Again, if $W$ has three or more vertices, then with the notation
from the proof of Theorem \ref{TPrim}, as before we see that
$\mathbf{s}_2\mathbf{s}_1$ and $\mathbf{s}_3\mathbf{s}_2\mathbf{s}_1$ are
two distinct right-primitive monomials. If $W$ has two nodes, then
suppose $v_1$ has order $d\geqslant 3$. Then $\mathbf{s}_1\mathbf{s}_2$
and $\mathbf{s}_1^2\mathbf{s}_2$ are two right-primitive monomials.

\item If $W$ is of type $H_3$, suppose the vertices are $v_1,v_2,v_3$ in
that order with $m_{v_1,v_2}=5$. Then
$\mathbf{s}_1\mathbf{s}_2\mathbf{s}_1\mathbf{s}_2\mathbf{s}_3$ and
$\mathbf{s}_3\mathbf{s}_1\mathbf{s}_2\mathbf{s}_1\mathbf{s}_2\mathbf{s}_3$
are both right-primitive.

\item Suppose the two vertices are $v_1,v_2$ with $d_1=3$ and $d_2=2$.
Then $\mathbf{s}_1\mathbf{s}_2$ and
$\mathbf{s}_1\mathbf{s}_1\mathbf{s}_2$ are right-primitive.

\item This case gives a commutative algebra, so all right-primitive
elements are primitive. \qedhere
\end{enumerate}
\end{proof}

With a bulk of the work done in the proofs of the above theorems, we
conclude by showing the final outstanding main result.

\begin{proof}[Proof of Theorem \ref{TFrobenius}]
The equivalence $(2)\iff (4)$ follows from Theorem \ref{TRightPrim}. Let
us now prove $(1)\implies (2)$ by extending the argument used to show
\cite[Theorem 5.2]{Kh2}.

Suppose $\nilhecke{<k}$ is Frobenius; so there exists a nondegenerate
invariant bilinear form $\sigma$ on $\nilhecke{<k}$. Now for each
non-zero primitive $p$, there is $a_p\in\nilhecke{<k}$ so that $0\neq
\sigma(p,a_p)=\sigma(pa_p,1)$. Now if $a_p\in\mathfrak m$, $pa_p=0$, so
one may assume $a_p=1$ for all $p$. Now the linear functional
$\sigma(-,1)$ gives an injective homomorphism from the set of
right-primitive elements to $\bk$: indeed, if we have $a\neq b$, both
right-primitive, so that $\sigma(a,1)=\sigma(b,1)\iff \sigma(a-b,1)=0$,
then we claim that $\sigma(a-b,c)=0$ for all $c\in\nilhecke{<k}$.
By linearity, it suffices to prove this for $c=1$ and $c\in\mathfrak m$.
If $c=1$, the conclusion is clear; if $c\in\mathfrak m$,
$\sigma(a-b,c)=\sigma(ac-bc,1)=\sigma(0,1)=0$ since $a,b$ are
right-primitive. This contradicts the non-degeneracy of $\sigma$, and
thus injectivity must hold.

Now this clearly implies the space of right-primitive elements is at most
one-dimensional. Since it has dimension at least one (for example, the
longest non-zero monomial is necessarily right-primitive), the dimension
is exactly one.

We now show $(4)\implies (1)$. Indeed, in the first case in (4),
$\nilhecke{<k}=\nilhecke{<\infty}$ is the usual nil-Coxeter algebra over
$W$, and it is Frobenius by \cite{Kho}. In the second case in (4),
$\nilhecke{<k} = \bk[\mathbf s]/\langle \mathbf s^d\rangle$ was shown to
be Frobenius in \cite{Kh2}, e.g.\ use the bilinear form $\sigma$ obtained
via $\sigma(\mathbf s^i,\mathbf s^j)=\mathbf{1}(i+j=d-1)$.

It remains to show $(2)\iff (3)$. The direction $(3)\implies (2)$ is
clear, so assume $(2)$. Since we have shown $(2)\iff (4)$, one simply
needs to verify $(3)$ holds in the cases listed under $(4)$. The first
case is handled in \cite{Kh2}, while the second case is trivial to check.
The proof is now complete.\end{proof}



\appendix
\section{Sage codes}
Here, we include the \textit{Sage} programs for verifying some of the
results that were obtained computationally.

The following calculates $\dim\nilhecke{<4}$ for $F_4,H_3$ and $H_4$. To
increase efficiency, this algorithm uses the fact that the set of all the
group elements corresponding to the monomials to be enumerated form a
weak order ideal.

\begin{lstlisting}[breaklines]
F4,H3,H4=WeylGroup(['F',4]),CoxeterGroup(['H',3],implementation="coxeter3"),CoxeterGroup(['H',4],implementation="coxeter3")
checkF4=lambda w:all(['2323' not in "".join([str(i) for i in x]) for x in w.reduced_words()])
checkH3=lambda w:all(['23232' not in "".join([str(i) for i in x]) for x in w.reduced_words()])
checkH4=lambda w:all(['34343' not in "".join([str(i) for i in x]) for x in w.reduced_words()])
I1,I2,I3=F4.weak_order_ideal(predicate=checkF4),H3.weak_order_ideal(predicate=checkH3),H4.weak_order_ideal(predicate=checkH4)
print(I1.cardinality(),I2.cardinality(),I3.cardinality())
\end{lstlisting}

The following code returns a list of all primitive monomials in
$\nilhecke{<k}$ for $k=3$, $W=H_3$, thereby proving the space of such
elements is one-dimensional.
\begin{lstlisting}[breaklines]
W=CoxeterGroup(['H',3],implementation="coxeter3")
s=W.simple_reflections()
checkH3_FC=lambda w:all(['23232' not in "".join([str(i) for i in x]) for x in w.reduced_words()]) and all(['121' not in "".join([str(i) for i in x]) for x in w.reduced_words()])
FClist=[w for w in W if checkH3_FC(w)]
def primitive(w):
    l=w.length()
    return all([not checkH3_FC(w*i) or (w*i).length()<l for i in s]) and all([not checkH3_FC(i*w) or (i*w).length()<l for i in s])
for w in FClist:
    if primitive(w):
        print(w)
\end{lstlisting}
\end{document}